\documentclass[11pt]{amsart}
\usepackage[T1]{fontenc}
\usepackage[ansinew]{inputenc}
\usepackage{amssymb,amsmath,amsthm,eucal,mathrsfs,array,graphicx,xcolor,bbm}
\usepackage[all]{xy}
\usepackage{fullpage}

\newtheorem{thmm}{Theorem}
\newtheorem{lem}[thmm]{Lemma}
\newtheorem{rem}[thmm]{Remark}
\newtheorem{prop}[thmm]{Proposition}
\newtheorem{cor}[thmm]{Corollary}

\theoremstyle{definition}

\newtheorem{res}{Result}

\def \g {\gamma}
\def \d {\delta}

\def \e {\varepsilon}
\def \eps {\varepsilon}

\def \O {\Omega}
          
\def \CC {\mathbb{C}}
 
\def \E {\mathbb{E}}

\def \P {\mathbb{P}} 
 
\def \R {\mathbb{R}} 
\def \Z {\mathbb{Z}}

\def \C {\mathcal{C}} 

\def \G {\mathcal{G}}
\def \Ha {\mathcal{H}}

\def \St {\mathcal{S}}
\def \T {\mathcal{T}}

\def \Es {\mathrm{Es}}

\def \SLE {\mathrm{SLE}}

\def \SS {\mathcal {OS}}
\def \cst {\iota}

\title{Near-critical spanning forests and renormalization}

\author{St\'ephane Benoist}
\address{Department of Mathematics,
Columbia University,
2990 Broadway,
New York, NY 10027, USA}
\email{sbenoist@math.columbia.edu}

\author {Laure Dumaz}
\address{
Statistical Laboratory, University of Cambridge,
Wilberforce Road,
Cambridge CB3 0WB, UK}
\email{L.Dumaz@statslab.cam.ac.uk }

\author {Wendelin Werner}
\address {D-Math, ETH Z\"urich, R\"amistr. 101, 8092 Z\"urich, Switzerland}
\email {wendelin.werner@math.ethz.ch}

\begin{document}

\begin {abstract}
We study random two-dimensional spanning forests in the plane that can be viewed both in the discrete case and in their appropriately taken scaling limits as a uniformly chosen spanning tree with some Poissonian deletion of edges or points. We show how to relate these scaling limits to a stationary distribution of a natural coalescent-type Markov process on a state-space of abstract graphs with real-valued edge-weights. This Markov process can be interpreted as a renormalization flow. 

This provides a model for which one can rigorously implement the formalism proposed by the third author in order to relate the law of the scaling limit of a critical model to a stationary distribution of such a renormalization/Markov process: When starting from any two-dimensional lattice with constant edge-weights, the Markov process 
does indeed converge in law to this stationary distribution that corresponds to a scaling limit of UST with Poissonian deletions. 

The results of this paper heavily build on the convergence in distribution of branches of the UST to SLE$_2$ (a result by Lawler, Schramm and Werner) as well as on the convergence of the suitably renormalized length of the loop-erased random walk to the ``natural parametrization'' of the SLE$_2$ (a recent result by Lawler and Viklund).
\end {abstract}

\maketitle

\section {Introduction}
Phase transitions and critical phenomena are now considered from the physics point of view to be a fairly settled issue, thanks to numerous important works in
these last 70 years.
On the mathematical side, there now exist a couple of important discrete two-dimensional models for which one can really prove that the discrete critical system  converges to a continuous scaling limit (that turns out to be conformally invariant), but many fundamental questions remain unsolved. This includes the existence and the description of scaling limits for three-dimensional models, and the understanding of the universality question (for instance: How can one prove that for a given model at criticality -- say critical percolation -- and a given dimension, it does behave in the same way in the scaling limit, independently of the chosen lattice?).

One of the arguments used successfully by physicists to tackle this universality question is that of the renormalization group. The underlying idea is that in the scaling limit, a critical system should give rise to a scale-invariant random continuous model. Then, if one manages to give a rigorous meaning to the change-of-scale operation as acting on these random configurations in the continuum (and each type of discrete model should then correspond to a different renormalization operation),
it has been argued that in fact, for every given spatial dimension $d$ and any critical model that gives rise to a random scaling limit, there should exist a unique non-trivial probability measure on continuous configurations that is invariant under this renormalization operation. Then, if one starts from the discrete model on any given
$d$-dimensional lattice and iterates this renormalization map (which corresponds to zooming out), one should converge to this unique critical continuous model.  

While this line of thought has sparked a number of important works on the field-theoretical description of these scaling limits, including mathematically rigorous ones, one 
major issue that mathematicians have not been able to circumvent is  to  make rigorous sense of the renormalization operation as
acting on some concrete geometrically-flavored state-space. 

The present paper's contribution is to implement in one very special case the renormalization formalism that has been described and proposed in \cite {W} for all critical FK-percolation models, and in any dimension. Recall that the critical FK-percolation models form a family of models indexed by $q \ge 0$ closely related to the Ising and Potts models, and that the cases $q=0$ and $q=1$ correspond respectively to the uniform spanning tree/forest model and to Bernoulli percolation. 
The general idea proposed in \cite {W} is to consider certain Markov processes living on the state-space of discrete weighted graphs. For each value of $q$, one can define 
one such Markov  process in rather simple terms -- it is a jump process, where jumps correspond to merging of neighboring sites (and the rate at which this happens depends on the graph and on $q$). Then, one can relate the existence of the scaling limit and universality question 
to some conjectural properties of these Markov chains, and more precisely to 
the existence of probability measures on such weighted graphs that are invariant under (a variant of) this Markov chain. 
We are not going to repeat here the description of this framework in the general case and we refer the reader to \cite {W} for details. As mentioned
in \cite {W}, in the special case of two-dimensional Bernoulli percolation, the detailed results of Garban, Pete and Schramm \cite{GPS1,GPS2,GPS3} on the phase transition and the near-critical behavior (which in turn partially build on Smirnov's conformal invariance results and/or the SLE$_6$ description of the critical interfaces) does 
provide a construction of such a non-trivial invariant probability measure for $d=2$ and the Markov process corresponding to Bernoulli percolation. 

In the present paper, we will focus of the special critical FK model for which one arguably has currently the most mathematical control on, 
namely the two-dimensional uniform spanning tree (UST) corresponding to $q=0$. Indeed, in this case, all the following features are known:  
 Existence of the scaling limit, its description (via SLE curves), universality (i.e., USTs on different lattices have the same scaling limit) \cite {LSW_LERW}, 
 and some very precise asymptotic estimates on probabilities. In particular, this is the only model for which the 
 appropriately renormalized lengths of interfaces in the discrete models are known to converge to the natural parametrization of its SLE scaling limit (see \cite {LV} 
 and the references therein -- this is in particular closely related to Kenyon's results \cite {Kenyon}).  We will make an extensive use of all these features. 
 
Let us first quickly describe the corresponding Markov process when started from a given infinite graph (one can 
for instance choose the starting point of the Markov process to be a two-dimensional lattice, but the set-up can be adapted to any dimension). 
First, imagine that one samples a UST on this graph, and then discovers its edges in a uniformly chosen random order (for instance, each edge that is eventually in the UST appears independently at some random exponential time). 
In this way, at a given time $t$, one has already some partial information about the UST. More precisely, the configuration is that of a forest $F(t)$ (a collection of trees that will all eventually be part of the infinite UST at time $t = \infty$). We can then consider for each time $t$, the graph $S(t)$ obtained by contracting all
edges that are present at time $t$ and that we will refer to as the structure graph $S( F(t))$ of 
$F(t)$.
More precisely, to each forest $F$ in our original lattice,  $S(F)$ is a graph with integer edge-weights defined as follows: 
\begin {itemize}
 \item Clusters $c$ of $F$ correspond in a one-to-one way to sites $s(c)$ of $S(F)$.
 \item When two clusters $c$ and $c'$ are not adjacent, there is no edge joining $s(c)$ and $s(c')$ in $S(F)$.
 \item When two clusters $c$ and $c'$ are adjacent, then $s=s(c)$ and $s'=s(c')$ are joined in $S(F)$ by an edge $(s,s')$ with weight $w(s,s')$ equal to the number of edges of the original lattice that connect $c$ to $c'$.
 \end {itemize}

\begin{figure}[ht!]
\centering
\includegraphics[width=14cm]{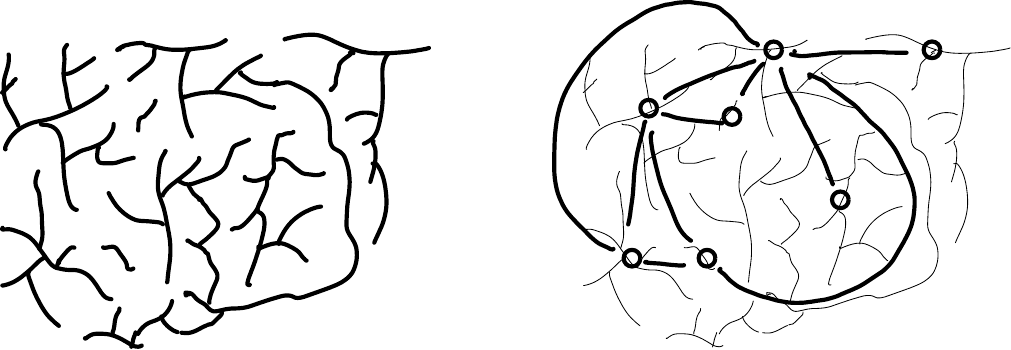}
\caption{From the forest to the structure graph (sketch)}
\end{figure}

It is easy to check that the process $(F(t), t \ge 0)$ is a Markov process (loosely speaking, the conditional distribution of the UST given $F(t)$ is just ``a UST conditioned to contain $F(t)$'').  A first key observation is that (when defined on an appropriate state of infinite weighted graphs), the process $(S(t):= S(F(t)), t \ge 0)$ is Markovian as well (this is just because the only information about $F(t)$ that is used to describe the future evolution of the structure graph is encapsulated in $S(t)$): 
The time-evolution for $S(t)$ then corresponds to the merging of neighboring sites $s'$ and $s$ (i.e., collapsing of the edge $(s,s')$ between them) that occurs at a certain rate depending on all the weights $w$ (i.e., it depends on the entire graph $S(t)$). When one collapses $s$ and $s'$ into a site $ss'$, then the new edge-weights 
$\tilde w$ are simply given by $\tilde w (ss' , s'') = w(s, s'') + w (s', s'' )$, while the edge-weights corresponding to edges that are not adjacent to $ss'$ are left untouched. 

A second observation is that if one multiplies all edge numbers by the same constant factor, then the only effect on the evolution of the Markov process is that it gets speeded up by a constant factor as well.
This leads naturally to consider the generalization of the Markov process on a space weighted graphs $S$, where the weights $c$ are non-negative reals (instead 
of integers). 
It is then also natural, for each positive $\lambda$, to consider the same Markov process but where all the weights decrease continuously at constant rate $\lambda$, in order to balance the general increase 
of edge-weights due to the constant contraction of edges.

We can now describe in loose words the content of the main results (Theorems \ref {main1} and \ref {main2}) of the present paper: 
\begin {enumerate} 
\item 
We will first see that the definition of our Markov process $(S(t), t \ge 0)$ on discrete structure graphs can be extended to a
space of graphs with unbounded degrees. Here, a site $s$ of $S$ can have infinitely many neighbors, but the sum of all weights $w(s, s')$ over all the neighbors $s'$ has to be finite.  
\item 
Using the continuous SLE-based description of the scaling limit of the two-dimensional UST, we will exhibit a non-trivial probability measure that is invariant for this Markov process for some positive value of $\lambda$.
This will build on the description of the scaling limit of UST via SLE, and on the convergence of the renormalized length of these branches to their continuous counterparts.
\item
Hence in the case of the two-dimensional UST, this implies that the conjectures for the formalism introduced in \cite {W} hold:   
when one starts from any two-dimensional lattice and runs the Markov process (for this value of $\lambda$), 
it converges in distribution to this particular fixed point of this Markov process
(up to multiplication of the weights of all edges by some lattice-dependent constant). 
\end {enumerate}

In order to describe the invariant probability measure under the Markov process and also to get a feeling about the strategy of the proofs, it is useful to look 
at the dynamics ``backwards'': one first samples the whole UST, and then for each time $t$, one creates the  
{\em uniformly cut uniform spanning tree} by erasing some of its edges uniformly at random, in a Poissonian way where each edge is 
removed with a probability $e^{-t}$ (we do this in a consistent way, so that an edge erased at time $t$ is also erased at all times $t' \le t$). 
So, at $t=0$, one has isolated points, and at $t= \infty$, one has the entire UST. 
Using the known convergence of the UST (in the scaling limit) to the continuous UST described in terms of SLE$_2$ and the convergence of the lengths of branches, one can argue that when $t$ is very large, the picture  (in
the appropriate scaling) will be very close to that of the continuous UST (constructed via Schramm-Loewner-Evolutions of parameter $2$) where the branches of the tree are cut in a Poissonian way with respect to their natural length. An invariant probability measure under the dynamics will be the law of this {\em uniformly cut continuous UST}, or more precisely the law of the weighted structure graph obtained by 
considering as sites $s$ the connected components of this uniformly cut continuous UST, and as weight $w(s,s')$ of the edge between two neighboring components the ``natural length'' of the interface between $s$ and $s'$. 
The following feature (that also appears in the work on near-critical percolation) of the stationary measure is worth stressing, as it illustrates the type 
of problems that one is facing. 
Consider a uniformly cut continuous UST, and two of its adjacent trees. Then, the appropriately defined $5/4$-dimensional length of the intersection between the 
boundaries of these two tree is comparable to (i.e., of the same order of magnitude as) the outer boundary of these trees, but this interface is in fact totally disconnected. There will be a dense collection of other small trees that are squeezed in between the two.

\medbreak 

\begin{figure}[!h]
\centering
\includegraphics[width=10cm]{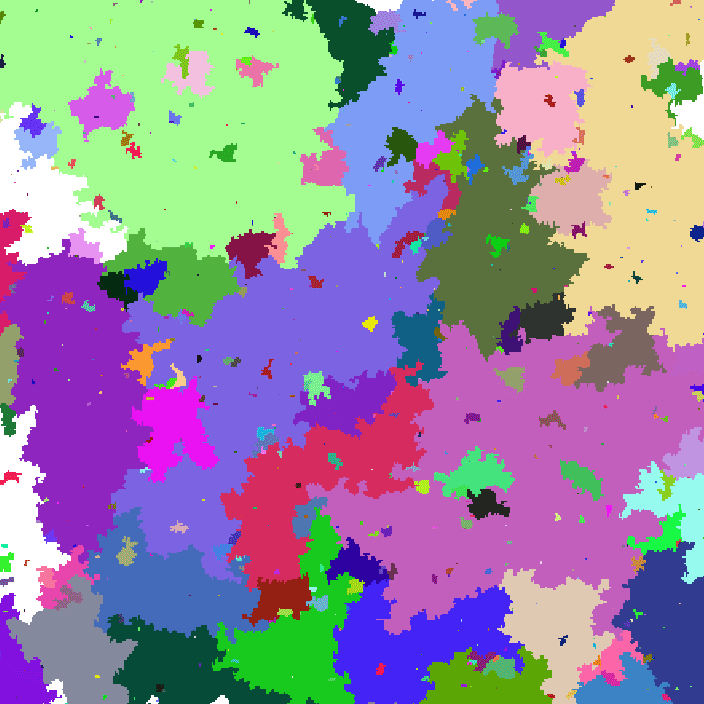}
\caption{Simulation of part of the uniformly cut UST}
\end{figure}

Here is a list of some of the main technical features and tools that we shall use: 
\begin {itemize}
 \item We will use the framework introduced by Schramm \cite {S0} in order to describe the set in which our discrete geometric objects (the UST, the uniformly cut USTs) and their scaling limits live in:  one encodes the limit of these forests to be the (countable) family of its ``continuous backbone branches'' (corresponding to the limit of the macroscopic branches of the cut UST).
 \item On this set of continuous forests, we will then define the dynamics. While the discrete dynamics are clearly Markovian, it is not obvious at all that the continuous process is 
 Markovian as well (as some information may have disappeared in the scaling limit). This is the same key-problem as in the case of near-critical percolation studied in \cite {GPS1,GPS2,GPS3}. In order to prove this, we need a careful analysis of the discrete to continuous limiting procedure, and we shall use some stochastic comparisons between the evolutions of various graphs under our dynamics.
 \item 
We rely on the convergence of the branches of the UST (i.e., loop-erased random walks converge to SLE$_2$, as proved in Lawler, Schramm, Werner \cite {LSW_LERW}) but one also needs to control the clocks of our dynamics, i.e., the number of edges on these branches, as they control the time-evolution.
For this, we will in fact use the convergence of the loop-erased random walk to SLE$_2$ in this ``natural parametrization'' for the uniform topology
(this result will be recalled in the next section), due to Johansson-Viklund and Lawler \cite {LV} (and that builds on earlier work of these authors with Benes \cite {BVL}). 
\end {itemize}
The paper will be structured as follows: 
\begin {itemize}
 \item In Section \ref{sec:USTandUSTlimits}, we first recall some features of USTs, briefly define Schramm's framework for scaling limits, and investigate the scaling limit of the cutting process of USTs in bounded domains.
 \item In Section \ref{sec:forw}, we study the time-reversal of the cutting dynamics seen on ``structure graphs'', and state our first main result, i.e., that this time-reversal is Markovian. We then explain why the whole-plane version of these results can be interpreted in terms of a renormalization flow fixed point. 
 \item In Section \ref{lastsection}, we prove the technical lemmas on discrete UST events that are needed in the previous proofs.
 \item In the appendix, we use the results of \cite {LV} in order to derive the actual facts about convergence of length to natural parametrization in the settings that we need.  
\end {itemize}

Let us conclude this introduction with a few words about ``near-critical'' models and stress that the uniformly cut USTs that we 
are working with here are not part of the FK-percolation family (this feature also appears in the general setup described in \cite {W}).
Recall that the terminology  ``criticality'' usually refers to the fact that one considers a one-parameter family of lattice models, and that there is a phase-transition at this special value of the parameter that one chooses. However, there are often more than one parameters that one can play with in the discrete model, and therefore, in the scaling limit one obtains many possible directions in which one can  perturb the continuous critical model as well.

On a finite graph, it is well known that the law $P^0$ of the uniform spanning tree can be viewed as the limit when $p \to 0+$ and $q = o(p)$ of the random cluster (or FK)-measure $P_{p,q}$ (indeed, the fact that $q \to 0$ faster than $p$ ensures that most of the mass of $P_{p,q}$ sits on the configurations with just one connected component, and the fact that $p \to 0$ ensures that the system uses the minimal amount of edges).
When $q \to 0$ and $p >0$ remains fixed, the measure becomes simply $P_{p, 0+}$, which is a percolation of density parameter $p$, conditioned to have exactly one connected component. On the other hand, when $q \to 0$ and $p$ is of the same order as $q$, then the limit will be supported on forests (i.e., collections of trees). More precisely, when $p=q$, the $q \to 0$ limit is the uniform measure on forests and when $p= \alpha q$, the limit measure is the percolation measure of parameter $\alpha/(1+\alpha)$ conditioned on the non-existence of open circuits. This leads (via finite-site scaling, tuning $\alpha (N)$ appropriately, and letting $q \to 0$ and $N \to \infty$) to a continuous model, which corresponds to a model of a near-critical continuous uniform spanning forest, which is a perturbation of the continuous UST. 
However, the object obtained via such a construction will differ from the one that we study in the present paper. One way to see this is to notice that a discrete measure $P_{\alpha 0+,0+}$ assigns the same probability to different forests that have the same number of trees, whereas in the uniformly cut UST, the weight of a configuration depends in a non-trivial way on the lengths of the boundaries between the trees in the forest (as they indicate how many possible ways there were to construct the forest by cutting the tree at random).

\section{UST and UST limits}\label{sec:USTandUSTlimits}

\subsection {General UST Background}
\label {USTbackground}
Let us very quickly browse through some of the standard UST features and definitions that we will use. 

The uniform spanning tree (UST) $\T(G)$ on a finite connected graph $G$ is a random subgraph of $G$ that has been uniformly chosen among those connected subgraphs that contain all vertices of $G$ and are cycle-free.
If $G$ is an infinite graph, one can define a similar object $\T(G)$, the free uniform spanning forest or free USF (see, e.g., \cite{BLPS}), as the weak limit of USTs on $G_n$, where $G_n$ is any increasing exhaustion of $G$ by finite connected subgraphs. Depending on the infinite graph, this uniform spanning forest can be almost surely a tree, or not. 
In $\Z^d$ for $d \le 4$, the free uniform spanning forest is actually a.s. a tree (and called a UST as well). 

The notion of UST can be extended to weighted graphs.
Let $G = (V,E)$ be a finite graph and let $c:E \to \R_+$ denote its weights. Then the weighted spanning tree is the probability measure on the set of all spanning trees such that the probability to choose a tree $T$ is proportional to $\prod_{e \in T} c(e)$. 
If $G$ is an infinite weighted graph, one can define the \emph{weighted free spanning forest}, a probability measure on the subgraphs of $G$, as the weak limit of the weighted spanning tree on $G_n$ (where $G_n$ is any connected exhaustion of the weighted graph $G$). This definition in fact works even if the graph is not 
locally finite (i.e., sites are allowed to have infinitely many neighbors, and the sum of the incoming weights is even allowed to be infinite). 
Depending on $G$, this weighted free spanning forest can almost surely be a tree or not.

\medbreak

Suppose now that $T$ is a spanning tree of the graph $G$ and that $V$ is a finite set of vertices $G$. We  will 
denote by $T_V$  the minimal connected subgraph of $T$ containing $V$. 
If $F$ is a forest (a disjoint union of trees) of $G$, then we define $F_V$ as the union of the subtrees generated by $V$ on all the connected components of $F$.

Wilson \cite{Wi} provided an algorithm to sample from the UST measure on a finite graph $G$, by iteratively generating branches as loop-erased random walks on the graph $G$ as follows. 
Enumerate the vertices of the graph $G$ as $x_0,x_1,\ldots, x_N$.  
Start with a single point $T_0 = x_0$. For each $n \in \{1, \ldots, N \}$, in order to build $T_n$, run a simple random walk $X_n$ on $G$ started from $x_n$ and stopped upon hitting $T_{n-1}$. Consider the (chronological) loop-erasure $\gamma_n$ of $X_n$, and let $T_n=T_{n-1}\cup \gamma_n$. Then, the final tree $T_N$ has the law of a UST on $G$.

It is well-known that Wilson's algorithm can be extended to (locally finite) infinite graphs such as $\Z^d$, as well as to weighted graphs (one just needs to replace the simple random walk by a random walk with non-constant conductances). This generates a random infinite forest, known as the wired spanning forest. In $\Z^d$ or in graphs that are obtained from $\Z^d$ by contracting or erasing some edges, the free USF and the wired USF coincide, see \cite {BLPS}.

\medbreak

At some points in the paper, we will use coupling results between USTs in various domains (this type of result is in fact instrumental in deriving the existence and properties of some of the objects mentioned above, such as the free USF). 

Let us first recall (\cite[Corollary 4.3-(a)]{BLPS}) that if one considers two connected graphs $G$ and $G'$ with the same vertex sets, but where the set of edges of $G$ contains the set $E'$ of edges of $G'$, then it is possible to couple the UST in $G$ with the UST in $G'$ in such a way that $\T (G) \cap E' \subset \T (G')$ almost surely. 

Suppose now that $\mathcal{I}$ is a collection of edges of a finite graph $G$, and let $\mathcal{I}_1 \subset \mathcal{I}_2$ be two subsets of $\mathcal{I}$ that can be both completed into spanning trees of $G$ by adding edges that are not in $\mathcal{I}$. 
Let $\T_1$ (resp. $\T_2$), be the uniform spanning tree $\T(G)$ on $G$, conditioned on $\T \cap \mathcal{I} = \mathcal{I}_1$ (resp. on $\T \cap \mathcal{I} = \mathcal{I}_2$).
It is then possible to couple $\T_1$ and $\T_2$ in such a way that $\T_2\cap\mathcal{I}^c \subset \T_1\cap\mathcal{I}^c$ almost surely.

Indeed, one can first condition 
 both USTs  to contain all edges in ${\mathcal I}_1$ and no edge in ${\mathcal I} \setminus {\mathcal I}_2$ (and this corresponds to just removing the edges 
 of ${\mathcal I} \setminus {\mathcal I}_2$ from the graph and to collapse all edges of ${\mathcal I}_1$). Hence, one needs only to treat the case where ${\mathcal I}_1$ is empty and 
 ${\mathcal I}_2={\mathcal I}$, which can be deduced from the previously mentioned result by conditioning on ${\mathcal I}\cap \T_1$.

Again, these results have fairly obvious generalizations to the case of weighted graphs (we safely leave their proofs to the readers).

\subsection {Schramm's framework}\label{subsec:Schramm}

In order to describe the scaling limits of our forests, we will use the framework introduced by Oded Schramm \cite{S0}; let us briefly review its basic features (we refer to Section 10 of \cite {S0} for details). 

For a compact topological space $X$, let us call $\Ha(X)$ the set of compact subsets of $X$ equipped with the Hausdorff topology; recall that 
  $\Ha (X)$ is itself a compact space. 

We call \emph{Schramm space} $\SS$ in the Riemann sphere $\hat \CC$ the set $\Ha(\hat \CC \times \hat \CC \times \Ha({\hat \CC}))$ equipped with its Hausdorff topology. 
Similarly, when $\Omega$ is a simply-connected bounded domain of the plane with $\C^1$ boundary, we define $\SS_\Omega=\Ha(\overline{\O}\times \overline{\O} \times \Ha(\overline{\O})) \subset \SS$.
The distance of this Hausdorff topology on $\SS$ or $\SS_\Omega$ is denoted by $d_{\Ha}$. 
Note that the notion of convergence is the same for the spherical or Euclidean distance in $\O$ when $\O$ is bounded, and so we can work with either in this case.
All these spaces are compact, so that any sequence of probability measures on those spaces possesses subsequential limits.

In the framework of uniform spanning trees and their scaling limits, one considers very special elements in $\SS$ (in particular elements $\G$ in $\SS$  with the property that if $(a,b,K) \in \G$,
then $K$ is a continuous path from $a$ to $b$, and $(b,a,K) \in \G$ -- furthermore, if $x$ and $y$ lie on this continuous path $K$ and $\omega$ denotes the part of $K$ from $x$ to $y$, then $(x,y,\omega) \in \G$ as well). 
A discrete graph embedded continuously in the plane (in such a way that the edges correspond to actual paths in the plane) can be encoded by its path ensemble, i.e., by a point $\G$ in the Schramm space such that $\G = \bigcup \{ (a,b, \g)\} $ where $a,b$ run over all pairs of points in the continuous embedding of 
the graph (so that $a$ and $b$ could lie on its ``edges'') and $\g$ runs over all simple (continuous) paths joining $a$ to $b$ in this embedded graph).
In particular, when $a$ or $b$ does not belong to the embedded graph or if $a$ and $b$ are in different connected components of this embedding, then there is no triplet of the form $(a,b,\g)$ in the corresponding path ensemble.

The UST on a discrete graph embedded in the plane can then be viewed as a probability measure on $\SS$, and by compactness, it has subsequential limits when one lets the mesh of the lattice go to zero. As we shall now recall, this subsequential limit is in fact a limit.

From now on and until further notice, $\Omega$ will denote either the entire plane or a simply-connected bounded domain of the plane with $\C^1$ boundary. We set $\Omega^\delta$ a simply connected discretization of it at mesh size $\delta$ of the same type as in \cite{LV} (``union of squares'' domain, paragraph 2.1):
we first consider the subgraph $A$ of $\delta \Z^2$ whose edges are exactly the edges of $\delta \Z^2$ that are included in $\Omega$. We then fix $\xi \in \Omega$ and let $\Omega^{\d}$ be the connected component of  $A$ that surround $\xi$. The boundary $\partial\Omega^\delta$ of $\Omega^{\d}$ will be the set of vertices of $\Omega^{\d}$ that have a neighbor that does not belong to $\Omega^{\d}$. As an illustration of discretizations, when $\Omega$ is the entire plane, we just take $\Omega^\d$ to be $\delta \Z^2$.

Let us consider the UST $\T(\Omega^\d)$ and its path ensemble denoted by $\G^\d(0) \in \SS_\Omega$
(we will soon run a dynamics starting from the UST at time $0$).
The branches of uniform spanning trees are loop-erased random walks (LERW), which have been shown by Lawler, Schramm and Werner to converge to SLE$_2$ paths in the scaling limit (this convergence holds for paths parametrized by ``Loewner capacity'', which yields in particular convergence for 
paths up to monotone reparametrization), \cite[Theorem 1.1]{LSW_LERW}. 

As explained in \cite {S0}, the convergence of LERW to SLE$_2$, together with estimates building on Wilson's algorithm, yield the convergence of the UST to its continuous limit in the Schramm space (we will refer to results and statements that are proved in other papers or preprints as ``results'' in order to make the distinction with 
the lemmas and propositions that are proved in the present paper): 

\begin {res}\label{res:CVtime0} (\cite[Corollary 1.2]{LSW_LERW} and \cite [Theorem 11.3]{S0}). When $\d \to 0$,  $\G^\d(0)$ converges in distribution (in $\SS_\Omega$) to a continuous random element $\G(0)$. 
\end {res}

There exists other possible descriptions of the scaling limits of USTs (for instance via the contour process of the tree, which converges to SLE$_8$, or via a consistent collection of subtrees \cite{ABNW}) but we will not use them here.
We will just call the random object $\G(0)$ the continuous UST in $\Omega$.
Theorem~1.5 of \cite{S0} lists various properties of $\G(0)$ (that for instance explain why one can call it a random tree). 
In particular, for every given $x,y \in \overline{\O}$, there exists almost surely a unique $\omega \in \Ha(\overline{\O})$ such that $(x,y,\omega) \in \G(0)$.
Moreover, if $x \neq y$, then $\omega$ is almost surely a simple path, and if $x =y$, then $\omega$ is almost surely a single point. 
There are some random exceptional points, for which this uniqueness statement does not hold (these points are nonetheless well-understood, in terms of the dual tree). However, existence almost surely never fails, i.e., almost surely, for any $x$ and $y$, there 
exists at least one  $\omega \in \Ha(\overline{\O})$ such that $(x,y,\omega) \in \G(0)$.

\medbreak

There are several ways to approximate the continuous UST in the Schramm space by somewhat simpler (continuous) objects. It is for instance natural to consider a dense deterministic sequence of points 
$z_1, z_2, \ldots$ in $\O$ and to define for each $n$ the finite subtree $\T_{z_1, \ldots, z_n}$ consisting of just the branches that join $z_1, \ldots, z_n$ (we have seen that they are almost surely unique). When dealing with such a tree in the Schramm space, we will implicitly consider the collection of all its subarcs (so in particular, for all $x$ and $y$ that lie on a branch of the tree, if $\omega$ denotes the branch of the tree from $x$ to $y$, then 
$(x,y,\omega)$ does belong to this typically uncountable collection). 
One key property derived in \cite {S0} is that when $n\to \infty$, this finite tree almost surely converges to the continuous UST in $\SS$. This statement holds in a strong way: the finite trees approximate well the entire tree in the sense that for all $\e >0$, with probability that goes to $1$ as $n \to \infty$, the whole tree is formed of the finite tree $\T_{z_1, \ldots, z_{n_\e}}$ plus some paths of diameter smaller than $\e$ with respect to the spherical metric in the plane. Let us state this more precisely.
 
\begin{figure}[ht!]

\includegraphics[width=10cm]{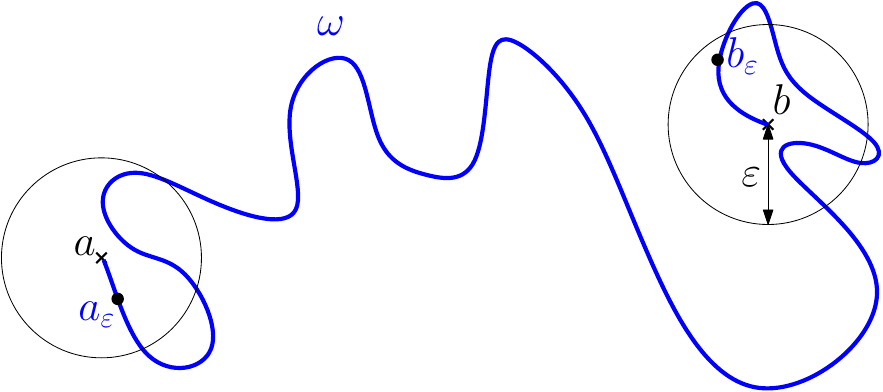}

\caption{An approximation of the branch $\omega$ between $a$ and $b$ by the branch $\omega'$ from $a_\e$ to $b_\e$ in $\G$.}\label{fig:ba}
\end{figure}

In what follows, when $\Omega$ is the entire plane, we will use the spherical distance (when $\Omega$ is bounded, one can safely use the Euclidean distance). 
We say that a subset $\G_\e$ of some $\G \in \SS_\Omega$ is a strong $\e$-approximation of $\G$ if for any point $(a,b,\omega) \in \G$, 
we can find $a_\e$, $b_\e$, $\omega_a$, $\omega_b$ and $\omega'$ such that $d(a,a_\e) \leq \eps$, $d(b,b_\e) \leq \eps$, $(a_\e,b_\e,\omega') \in \G_\e$, $\omega_a \subseteq B(a,\e)$, $\omega_b \subseteq B(b,\e)$ and $\omega = \omega_a \cup \omega' \cup \omega_b$ (see Figure \ref{fig:ba}). When $\G$ encodes the branches of a tree, approximations of this kind can be found by somehow removing the part of the branch $(a,b,\gamma_{a,b})$ in an ``$\e$-neighborhood'' of $a$ and $b$:
in \cite {S0}, Schramm defined the $\eps$-trunk as a subtree of the UST where the part of the branches that are $\eps$ close to the leaves are removed. It is then obvious that the $\e$-trunk is a strong approximation of $\G(0)$. Note that in particular the distance between $\G_\e$ and $\G$ is smaller than $\e$. 
The following result is a key step in \cite{S0} towards the proof of Result \ref{res:CVtime0}.

\begin{res}\cite[Theorem 10.2]{S0}\label{th.finiteness.S2}
For any cut-off $\e >0$, we can find a scale $\delta_\e>0$, such that for any mesh size $\delta < \delta_\e$ and for all
set of vertices $(z_1,\ldots,z_n)$ being a $\delta_\e$-net of $\Omega$ (i.e., every point in $\Omega$ is within distance $\delta_\e$ of one of the $z_i$), then the collection of all subarcs of the finite tree $\T_{z_1, \ldots, z_n}(\O^\d)$ generated by $z_1,\ldots,z_n$,  viewed in the Schramm space $\SS_\Omega$, is a strong $\e$-approximation of $\G^\d(0)$ with probability greater than $1-\eps$.
\end{res}

As a consequence (see \cite[Corollary 10.3]{S0}), for all $\e >0$, there exists $n$ such that the subtree $\T_{z_1, \ldots, z_n}$ of the continuous UST $\G(0)$ on $\Omega$ is a strong $\e$-approximation of $\G(0)$, with probability greater than $1-\e$.

\medbreak
\noindent
{\bf Dual trees and boundary conditions.}
It is well-known that for a planar graph (i.e., embedded in the plane so that no two edges cross), one can associate to each spanning tree $T$ on the graph $G$ a dual spanning tree on the dual graph, and that 
if $\T$ is sampled according to the UST measure, then the dual tree is sampled according to the UST measure in the dual graph. 
When $G$ is a portion of the lattice $\Z^2$, then the dual graph is a portion of the lattice $(\Z + 1/2 )^2$, with the boundary vertices identified (this corresponds to wired boundary conditions). 
In the discrete case, one can define $\G^\d(0)^\star$ in the Schramm space as being the dual tree of $\G^\d(0)$ (i.e., the element in the Schramm space corresponding to the dual of the tree $\T(\Omega^\d)$). 
By taking subsequential limits, one can then have convergence in distribution of the couple $ (\G^\d(0), \G^\d(0)^\star)$. It is explained in \cite {S0} that in fact, the limit of $ \G^\d(0)^\star $ is a deterministic function of the limit of $\G^\d (0)$.
We again refer to \cite {S0} for details (in particular about boundary conditions for the USTs) -- and for the fact that  Result \ref {th.finiteness.S2} holds also for USTs with wired boundary conditions.

\begin{rem}
Building on Wilson's algorithm, it is fairly easy to compare USTs with different boundary conditions, and to deduce the convergence (when the mesh size goes to $0$) of the UST in the entire plane from the convergence in bounded domains. For instance, if one considers $n$ points $y_1, \ldots, y_n$ in the plane, and the law of the finite tree $\T_n^\delta$ obtained by sampling the smallest subtree of the UST in $\delta \Z^2$ that contains $n$ points on this grid that are at distance smaller than $\delta$ from $y_1, \ldots, y_n$, then for all $\eps$ and $R$ large enough and $\delta$ small enough,  $\T_n^\delta$ is equal to the corresponding subtree $\widehat\T_n^\delta$ of the wired UST in the domain $\{ z  \ : \ |z| < R \}$, with probability greater than $1-\varepsilon$. However, the law of the tree $\widehat\T_n^\delta$ converges (when we let $\delta \to 0$ first and then $R \to \infty$) to a finite continuous tree $\T_n$ joining $y_1, \ldots, y_n$ (thanks to \cite[Theorem 1.1]{LSW_LERW} which holds for any simply connected domain). It follows that the tree $\T_n^\delta$ converges in law as $\delta \to 0$ to $\T_n$. We will use this approach later in Appendix \ref{sec:app} to get the strong convergence of UST for various boundary conditions.
\end{rem}

\subsection{UST and lengths of branches}

We now want to extend the previous convergence in distribution of the discrete UST to the continuous one, when one adds also the information about the lengths of the
branches of tree.
It is known since Rick Kenyon's paper \cite {Kenyon} that the mean number of steps of a LERW grows like $\delta^{-5/4 + o(1)}$ as the mesh-size $\delta$ goes to $0$
(see also \cite {Masson,BVL} for closely related sharper estimates and results). Note that the actual length of the LERW with mesh-size $\delta$ will grow 
like $\delta^{-1/4 + o(1)}$ because each edge has length $\delta$. 

On the other hand, it is also known (see \cite {Beffara}) that the scaling limit of LERW (i.e., SLE$_2$) is a random simple curve with Hausdorff dimension $5/4$.
In fact, it has been recently shown \cite{LR} that SLE$_2$ can be parametrized by its $5/4$-dimensional Minkowski content, (often referred to as the natural parametrization).
Recall that the $d$-dimensional
Minkowski content of a curve $\g$ is defined as:
\begin{align*}
 \mbox{Cont}_{d}(\gamma) = \lim_{\e \to 0} \e^{d-2} \mbox{Area}\{z \;:\; d(z,\gamma) \leq \e\}
\end{align*}
provided that the limit exists.

It is natural to expect that in fact, the suitably renormalized discrete length of the LERW should converge to the $5/4$-dimensional content of the limiting SLE$_2$. This non-trivial fact turns out to be correct:
Let $\Omega$ be a bounded simply connected domain with analytic boundary such that $0 \in \Omega$ and for each $\d$, recall that $\Omega^\d$ is a lattice approximation of $\O$ in $\d \Z^2$.
Consider a loop erased random walk starting at $0$ in $\Omega^\d$ (i.e., the loop erasure of a simple random walk stopped at the time $\tau^\d$ at which it hits $\partial \Omega^\d$), which we view as a continuous curve that takes one unit of time to cross an edge, and denote by $\gamma^\d$ its time-reversal.

The following result of \cite{LV2} 
will be an essential building block in our paper, which enables us to fine-tune the scale and control the cutting procedure.   
Here and in the rest of the paper, $\cst$ will denote a particular absolute constant (that can be viewed as a lattice-dependent constant -- it is here the constant associated to $\Z^2$; with other planar lattices, the same result would hold but with a different constant $\cst$).

\begin{res}[{\cite[Theorem 1.1]{LV2}}]\label{res:convLERW}
Let $\tau^\d$ be the total length of the path $\gamma^\d$. The curve $t \mapsto  \gamma^\d(\cst \, \delta^{-1/4} \min (t,\tau^\delta))$ converges in distribution to radial $\SLE_2$ 
curve $t \mapsto \gamma( \min (t, \tau))$ in $\Omega$ (starting from a point chosen with respect to the harmonic measure on $\partial \Omega$ seen from $0$) in its natural parametrization (where $\tau$ denotes its total natural length) for the topology of supremum norm.
\end{res}

In Appendix \ref {sec:app}, we will combine Result \ref{res:convLERW} with Wilson's algorithm to derive the following results:
\begin{itemize}
\item The convergence of the wired UST in a bounded domain $\Omega$, in its (properly renormalized) arc-length parametrization (Proposition \ref{cor:lengthwiredUST}).
\item The convergence, with their properly renormalized arc-length parametrizations, of the plane UST (Proposition \ref{prop:lengthplanarUST}) and the free UST (Proposition \ref{prop:lengthfreeUST}) -- the main difficulty lies in the proof of the statement for the free UST.
\item The convergence of the joint law of the UST and its dual, with their properly renormalized length parametrization (Corollary \ref{prop:joint}).
\end{itemize} 
In the remainder of this paper, when we will refer to the {\em appropriately rescaled} lengths of branches of discrete trees on $\delta \Z^2$ (or subgraphs of it), it will always
mean that one uses  $\delta^{1/4} / \cst$ times the Euclidean length parametrization (so that this appropriately rescaled length is the one that converges to the natural parametrization).

\subsection {Scaling limit of the cutting dynamics}

Recal that $\Omega$ is either the entire plane or a simply connected bounded domain with $\C^1$ boundary, and $\Omega^\d$ denotes its discretization at mesh size $\d$.

Let us now define the discrete cutting procedure. Let  $(-\tau_{e})$ be a family of i.i.d random exponential times with mean $\cst \d^{-5/4}$, indexed by the set of (non-oriented) edges $e$ of $\O^\d$.
We start at time $t=0$ with a UST $\G^\d(0)$ on $\O^\d$ independent of the family $(\tau_{e})$. For a fixed time $t < 0$, we define $\G^{\d}(t) \subseteq \G^{\d}(0)$ 
to be the spanning forest that is obtained from $\G^\d (0)$ by removing all the edges $e$ with $\tau_{e} \in (t,0]$ (viewed in the Schramm space, we remove all the paths that
go through at least one of these edges).
This defines a nested family of forests $(\G^{\d}(t))_{t \le 0}$. Note that the limit point $\G^{\d}(-\infty)$ is a graph without edges (encoded in the Schramm space by the collection $\{(v,v,\{v\}) : \ v \in \O^\d \}$).

Let us now define the continuous counterpart of this discrete cutting procedure. We first sample (for a given $\Omega$) the continuous UST $\T=\G (0)$. For any fixed $z_1, \ldots, z_n$, 
the $5/4$-Minkowski content of the tree $\T_{z_1, \ldots, z_n}$ is almost surely finite (this follows from Result \ref {res:convLERW} and our description of the scaling limit of the law of subtrees of the free UST in the appendix, as being absolutely continuous with respect to those of the wired UST scaling limit). We then sample a Poisson point process on this finite tree, so that marked points appear at negative times with an intensity given by the $5/4$-Minkowski content. As we do this simultaneously for any finite set of points $z_i$, we in fact are having marks appearing on the ``backbone'' of the continuous UST. 
We then define the continuous forest $\G (t)$ that corresponds to the continuous tree, by cutting all marked points that have appeared in the time-interval $(t,0]$.
\begin{figure}[ht!]
\centering
\includegraphics[width=7cm]{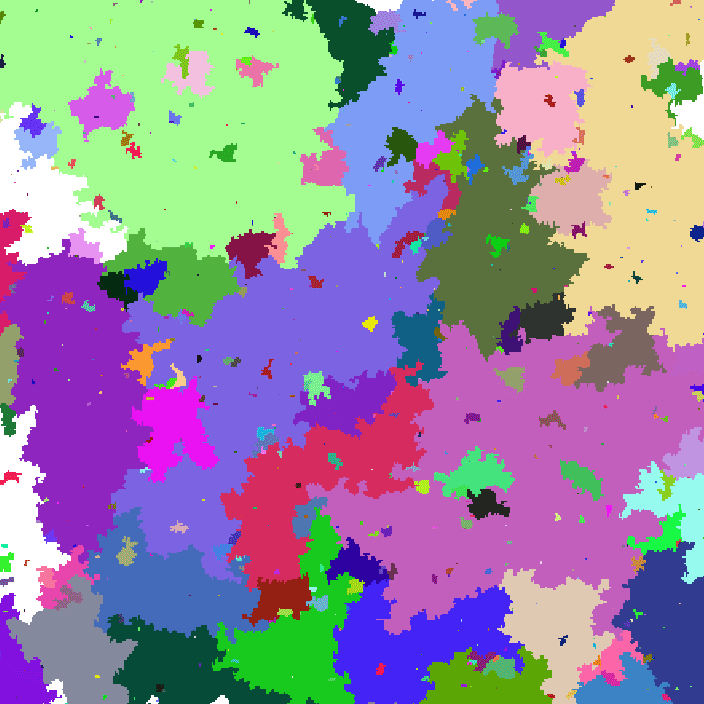}
\hspace{1cm}
\includegraphics[width=7cm]{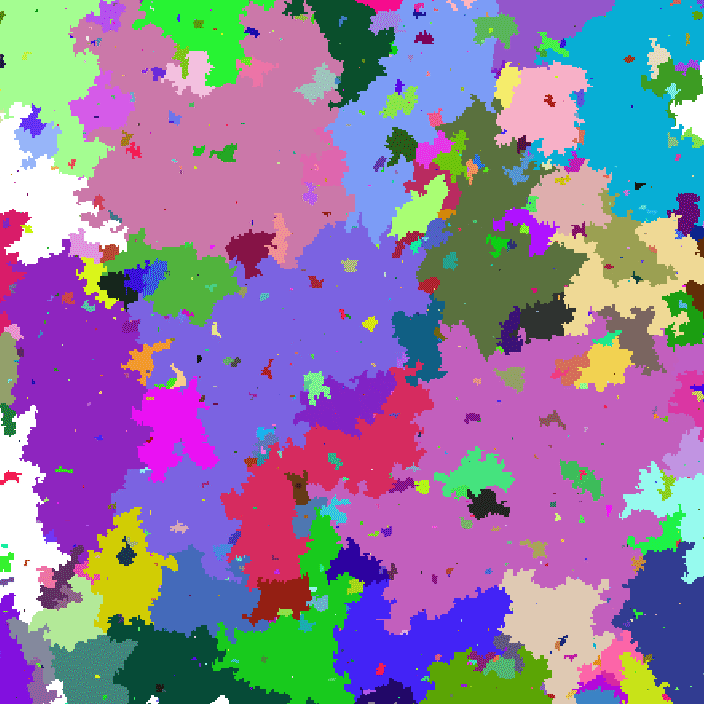}
\caption{Simulations of $\G^{\d}(-3)$ and $\G^\d(-4)$ (different clusters are indicated in different colors): the latter is obtained from the former by cutting, while the former is obtained from the latter via the glueing Markov process}
\end{figure}
 
Note that when $\Omega$ is the entire plane, the underlying metric used to define the Schramm space is the spherical metric, but the cutting procedure uses the $5/4$-dimensional content associated to the Euclidean metric, as it should correspond to the limit of the discrete length of the LERW on the graph.   

\begin{prop}\label{propo:CV_cutoff_process}
The process $(\G^{\d}(t))_{t \le 0}$ converges in distribution (in the sense of finite-dimensional distributions 
in $\SS_\Omega$) to the process $( \G(t))_{t \le 0}$.
\end{prop}

Note that the following proof will in fact establish a slightly stronger Skorokhod-type convergence on c\`adl\`ag processes.

\begin{proof}
We fix $\e,\eta>0$ and $t_0 < 0$ and our goal is to show that when $\d$ is small enough, one can couple the processes $(\G^\d(t))$ and $(\G(t))$ in such a way that on a set of probability at least $1 - 3 \eta$, for all $t \in [t_0, 0]$, $d_{\Ha}(\G^\d(t),\G(t)) \leq 3 \e$. 

We first find (using Result \ref {th.finiteness.S2}) a finite net $z_1, \ldots, z_n$  such that with probability greater than $1-\eta$, the finite subtree $\T_n := \T_{z_1, \ldots, z_n}$ generated by $z_1,\ldots,z_n$ is a strong $\e$-approximation of $\G(0)$, i.e., it differs from it by appending little trees of diameter less than $\e$ (by a slight abuse of notation, $\T_n$ will represent the tree both as a union of branches and as a point in Schramm space) and that for all $\d$ small enough, with probability greater than $1-\eta$, the finite subtree $\T_n^\d := \T_{z_1, \ldots, z_n}^\d$ generated by $z_1,\ldots,z_n$ is a strong $\e$-approximation of $\G^\d(0)$. 
In particular, to understand the cut forests $\G(t)$ (respectively $(\G^\d (t))$) up to a distance smaller than $\e$ and on an event of probability at least $1- \eta$, it will be sufficient 
to look at how $\T_n$ (resp. $\T_n^\d$) is being cut (the effect of additional cuts outside of $\T_n$ or $\T_n^\d$ would not move things in the Schramm space by more than $\eps$). 

Let us denote by $\T_n(t)$ the cutting process of the tree $\T_n$, i.e., the graph $\T_n\cap\G(t) \in \SS_\Omega$.
We similarly define  $\T^\d_n(t)$ the discrete cutting process of $\T_n^\delta$.
The tree $\T^\d_n$ can be divided into $n-1$ disjoint simple paths as in the way provided by Wilson's algorithm: $\gamma^\d_2$ denotes the branch from $z_2$ to $z_1$, and for all $k$ in $\{ 3, \ldots, n\}$, $\gamma_k^\d$ denotes the branch from $z_k$ to the subtree containing $z_1, \ldots, z_{k-1}$. Similarly, we can define $\gamma_2, \ldots, \gamma_n$ in $\T_n$. 

Propositions \ref{prop:lengthplanarUST} and \ref{prop:lengthfreeUST} tell us that the finite subtrees $\T^\d_n$, together with their appropriately rescaled length measure converge: any of the branches from $z_i$ to $z_j$ converges for the topology of supremum norm, to branches of the continuous tree $\G(0)$ in their natural parametrization. 
More specifically, when $\d$ is small enough, we can couple the trees $\T^\d_n$ and $\T_n$ in such a way that with probability at least $1- \eta$: 
(i)  for each $k \in \{ 2 , \ldots, n\}$, the total appropriately rescaled length of $\g^\d_k$ is $\eta/(|t_0 |n )$-close to the natural length of $\g_k$, and (ii)  for each $k \in \{ 2 , \ldots, n\}$, the two paths $\g^\d_k$ and $\g_k$ are uniformly $\eps$-close (in the sup-norm for those parametrizations, on the time-interval where they are both defined). 
 
We then couple the cutting dynamics in the discrete and in the continuum using the same exponential clocks: 
we sample $n-1$ independent Poisson point processes of intensity $|t_0|$ and we transfer these Poisson point processes onto the $n-1$ 
discrete and continuous branches using respectively the appropriately rescaled length and the natural parametrization (in the discrete setting, when at least one Poisson mark falls in an interval corresponding to an edge, we remove this edge). 
Condition (i) and (ii) ensure that when $\d$ is small enough, with a probability at least $1- \eta$, the number of 
Poisson marks that did fall in each branch $\g^\d_k$ and $\g_k$  will be identical, and that the location of these marks will be $\eps$-close. 

Putting the pieces together, we get when $\d$ is small enough, one has a coupling such that on a set of probability $1-3\eta$, for all time $t\in[t_0,0]$,
\[
d_{\Ha} ( \G^\d(t) , \G(t)) \le d_{\Ha}(\G^\d(t), \T^\d_n(t)) + d_{\Ha}(\T^\d_n(t),\T_n(t)) + d_\Ha(\T_n(t),\G(t)) \leq 3 \e.
\]
\end{proof}

\section{The structure graph and the scaling limit of the glueing dynamics}\label{sec:forw}

Let us now focus on the flow that one obtains when one looks at the time-reversal of the cutting dynamics on some interval $[t,0]$.

\subsection{Description of the discrete glueing dynamics} 

Recall that if we observe $\G^\d (t)$ for some given $t <0$, we can recreate the conditional law of $ ( \G^\d (s))_{ s \in [t,0]}$ in the following way. Denote by $n$ the number of connected components of $\G^\d(t)$. Let us pick uniformly a set of edges $E$ among all sets of $n-1$ edges of $\d \Z^2$ such that $\G^\d(t) \cup E$ is a spanning tree of $\Omega^\delta$. The graph $\G^\d(t)$ then evolves by iteratively gaining edges of $E$ (picked in uniform order), at the jump times of a Poisson process conditioned on jumping $n-1$ times in $[t,0]$ (or equivalently, edges of $E$ appear at independent uniformly chosen times).

Let us rephrase this evolution in a way that is more tractable in the continuum limit. We first (deterministically) associate to each $\G^\d(t)$ a structure graph $\St^\d(t)$ as described in the introduction:
Each connected component $c$ of $\G^\d(t)$ becomes a site of the structure graph $\St^\d(t)$. Two neighboring (and distinct) connected components are linked by an edge in the structure graph, which carries a positive weight equal to $\d^{5/4}$ times the number of edges in $\O^\d$ between the two connected components (edges with one end-point in each of the connected components).

By construction, the trace of the set of edges $E$ on the structure graph (which shows how the connected components of $\G^\d(t)$ are connected in the graph $\G^\d(0)$) has the law of the weighted spanning tree $\T(\St^\d(t))$ on $\St^\d(t)$. This describes the Markovian evolution of the discrete glueing dynamics when seen on structure graphs (each edge that is in the weighted tree then appear uniformly at random in the interval $[t,0]$).

Note that the conditional distribution of the evolution of $(\G^\d (s))_{s \in [t,0]}$ given the initial data $\G^\d (t)$ and the evolution of the structure graph $(\St^\d (s))_{s \in [t,0]}$ is easy to describe.
When two sites 
$c$ and $c'$ of $\St^\d (s^-)$ merge at time $s$, one recovers the graph $\G^\d (s)$ by adding to $\G^\d (s^-)$ an edge picked uniformly among the edges of $\d \Z^2 \setminus\G^\d (s^-)$ that join $c$ and $c'$.

 \subsection {Definition of the continuous structure graphs}
The first non-trivial job when trying to make sense of the continuous counterpart of this glueing dynamics on structure graphs is to construct the continuous structure graphs $\St(t)$.
For a point $z_0$ which does not lie inside a branch of the dual tree, let us formally define its connected component $c$ in $\G(t)$ as a subset of $\R^2$: $c$ is the closure of all the points $z$ such that there is a branch from $z_0$ to $z$ in $\G(t)$. Now, we would like the vertices of $\St(t)$ to be the connected components of $\G(t)$ and there should be an edge between two vertices of $\St(t)$ (i.e., connected components of $\G(t)$) whenever these components are not disjoint (i.e., whenever they share a piece of their boundaries).

The candidate  for the weight of these edges is (up to a constant) the $5/4$-dimensional Minkowski content of the interface between the corresponding clusters. Here we can note that this interface is in fact made of portions of branches in the dual tree, which suggests that we will need to control the lengths of the branches in the dual of the continuous tree.
This is the purpose of the next result (we defer its proof to Section \ref{lastsection}) that then defines, for each $t \le 0$, the weights of the structure graph $\St (t)$ and shows that they are indeed the limits of their discrete counterparts: 

\begin{prop}[Weights of the continuous structure graph]\label{lem:weightSt}
Consider two given points $z_0$ and $z_1$ in $\O$, and the connected components $c_0^\d(t)$ (resp. $c_0(t)$)  and $c_1^\d(t)$ (resp. $c_1(t)$) of $\G^\d (t)$ (resp. $\G(t)$) that they are part of, and let $l^\d (z_0, z_1)$ be the renormalized length  of the interface between $c_0^\d(t)$ and $c_1^\d(t)$ (respectively the $5/4$-dimensional Minkowski content $l(z_0, z_1)$ of the intersection between $c_0(t)$ and $c_1(t)$) when it exists.
Then, for each given $t$, the couple $(\G^\d (t), l^\d (z_0, z_1) )$ converges in distribution to $(\G(t), l (z_0, z_1))$. 
\end{prop}

Mind that this is not a trivial fact, because the structure graphs are rather complicated: we have to handle the infinitely many microscopic clusters appearing in the scaling limit and that will squeeze in between two macroscopic ones. One point in the proof (deferred to Section \ref{lastsection}) will be to control the effect of this feature.

\medbreak

In order to define the Markov dynamics on such structure graphs, we will need to define the (weighted) forests and trees on them. In order to do so, we will choose exhaustions $(\St_\e(t))_\e$ and $(\St^\d_\e(t))_\e$ of the graphs $\St (t)$ and $\St^\d (t)$. Recall that the limiting laws on forests (when $\eps \to 0$) do not depend on the choice for the exhaustions (see, e.g., \cite[\S 5]{BLPS}). In particular, we can choose exhaustions depending on the whole data of $\G^\d(t)$ (resp. $\G(t)$) as we see fit:
For all $\e >0$, we define the vertex set of $\St^\d_\e(t)$ (resp. $\St_\e(t)$) to be the subset of the vertex set of $\St^\d (t)$ (resp. $\St (t)$) consisting of the 
connected components  $\G^\d(t)$ (resp. $\G(t)$) that have a diameter at least $\e$ (when $\Omega$ is the entire plane, we use the spherical metric here).
The weighted edges between vertices of $\St^\d_\e(t)$ and $\St_\e(t)$ are then exactly those of $\St^\d (t)$ and $\St (t)$.

Note that the graph $\St_\e(t)$ is almost surely finite: indeed, by Result \ref{th.finiteness.S2}, we can almost surely find a strong $\e$-approximation of $\G(0)$ by a subtree $\T_n$, 
where $n$ is random but almost surely finite. 
The number of vertices of the graph $\St_\e(t)$ will then be not larger than the number of connected components of the forest $\T_n(t)$. It is also immediate to see that $(\St^\d_{\e}(t))_\e$ (resp. $(\St_{\e}(t))_\e$) exhausts $\St^\d(t)$ (resp. $(\St(t)$).

We now state the convergence of the structure graph. We use the discrete topology on finite graphs, and for a given finite graph, weights form a real vector space that we equip with its natural topology.
\begin{cor}[Discrete to continuous structure graph convergence]\label{thmm:cv-struct-graph}
For each $t <0$, for all but (at most) countably many positive $\eps$, the finite random graph $\St^\d_\e(t)$ converges in probability to $\St_\e(t)$ as the mesh size $\d$ goes to $0$.
\end{cor}

This results follows directly from Proposition \ref{lem:weightSt} (i.e., the convergence of the weights of the edges) and the convergence of $\G^\d(t)$ to $\G(t)$.
The values of $\eps$ we exclude here are those for which, with positive probability, there is a cluster in $\G(t)$ that is of diameter exactly equal to $\eps$. As we know that there are countably many clusters, it follows that this can happen (for each fixed $t$) for at most countably many $\eps$.
One could of course also (try to) prove that this never happens, but the present result will be enough for our purposes.


\subsection {Abstract definition of the Markovian dynamics on structure graphs} 
We are now ready to define the Markovian dynamics on the set of structure graphs. For a given $t$ and a given weighted graph $\St(t)$: 
\begin {itemize} 
 \item First, sample a weighted free spanning forest on $\St(t)$, and for each edge of this forest, sample independently a uniform random variable on $[t,0]$ that indicates when this edge appears.
 \item Then, construct the graph at time $s \in [t,0]$ by contracting all edges that have appeared before time $s$, and using the addition rule for weights: when two sites $s_1$ and $s_2$ merge into a site $s_1s_2$, the new weights are given by $w_{new} (s_1s_2, \cdot) = w_{old} ( s_1, \cdot) + w_{old} (s_2 , \cdot)$.
\end {itemize}
Recall that it is not a priori clear that the weighted spanning forest on the structure graph is a tree, but along our proof, we will see that in fact, it is indeed almost surely the case, when one starts this dynamics with the random graph $\St (t)$. Moreover, weights can blow up under the dynamics, depending on initial conditions. That this does not happen when we initiate our dynamics with the structure graphs of our near-critical spanning forests is a consequence of the following Theorem \ref{main1}.

In this way, one defines a process $(\tilde \St(s))_{s \in [t,0]}$, which is the evolution of this Markovian dynamics when applied to the random structure graph $\tilde \St (t) = \St (t)$. The core of the matter is then to prove the following fact: 

\begin {thmm} 
\label {main1}
 The law of $(\tilde \St(s))_{ s \in [t,0]}$ is the same as that of 
$(\St (s))_{ s \in [t,0]}$.
\end {thmm}

In loose words, the scaling limit of the Markov dynamics on discrete structure graphs is Markov, and it is described by the simple process on continuous graphs that we have described above. 
 Mind that the theorem is also valid when $\Omega$ is the full plane.

Note that, as in the discrete case, there is a (heuristically straightforward) description of the conditional distribution of $(\G(s))_{s \in [t, 0]}$ given $\G(t)$. Construct first $\St(t)$ and $(\tilde \St (s))_{s \in [t,0]}$. For each contraction of vertices $s(c)$ and $s(c')$ happening on $[t, 0]$, we choose a point $w$ according to the uniform measure on the common boundary of $c$ and $c'$, measured by its $5/4$-dimensional Minkowski content (this common boundary is the union of several portions of dual branches, and its content is well defined, as follows from Lemma \ref{prop_lengthdiscret}; the Minkovski content can be viewed as a proper measure when restricted to these branches -- this follows from Result \ref {res:convLERW} which implies that the Minkovski-content can be used as a continuous time-parametrization of these SLE$_2$-type paths). Let us call $\mathcal{W}(s)$ the countable set of points thus chosen that corresponds to contractions happening before time $s$. For each integer $n$, let $\tilde \G_n(s)$ be the union of the paths $(a,b,\gamma)$, such that $\gamma$ is a path from $a$ to $b$ that can be realized as the concatenation of at most $n$ paths in $\G(t)$, where the points of concatenation belongs to the set $\mathcal{W}(s)$. We 
then define $\tilde \G(s)$ to be the closure in $\SS$ of the union $\cup_n \tilde \G_n(s)$. It is easy to see that $(\tilde\G(s))_{s \in [t, 0]}$ has the same law as the limit of the discrete dynamics $(\G(s))_{s \in [t, 0]}$. Indeed each given branch $(a,b,\g_{a,b}) \in \G(0)$ is almost surely cut a finite number of times, and there almost surely exist a countable family of branches of $\G(0)$ that are dense among the set of all branches of  $\G(0)$ (see Result \ref{th.finiteness.S2}).

\medskip

Let us now explain how to deduce this theorem from the previous propositions. As we shall see, this is quite a soft argument, where we will exploit the tightness-type properties of the 
USTs (derived by Schramm) and coupling ideas.

\subsection {Proof of Theorem \ref {main1}}
Let us first recollect a few facts:
\begin {enumerate} 
 \item 
From  Result \ref{th.finiteness.S2}, we know that for a given $\eta$ and a given $\eps$, we can find a finite set of points $z_1, \ldots, z_n$, such that 
(for both the discrete case for all given $\delta$, and the continuous case), with probability at least $1 - \eta$, the connected components of $\G^\delta(t)$ (resp. $\G(t)$) corresponding to vertices of the graph $\St_\eps^\d (t)$ (resp. $\St_\eps (t)$) all intersect the tree $\T_{z_1, \ldots, z_n}^\d$ (resp. $\T_{z_1, \ldots, z_n}$).
\item On the other hand, for a given choice of $z_1, \ldots, z_n$, the convergence of the branches of the tree joining these points in their natural parametrizations ensures that one can find $\eps_1$ small enough so that (uniformly in $\d$, i.e., for each given $\d$) every connected component of $\G^\delta(t)$ (resp. $\G(t)$) that intersects the finite tree $\T_{z_1, \ldots, z_n}^\delta$ (resp. $\T_{z_1, \ldots, z_n}$) has diameter at least $\eps_1$, and hence corresponds to a vertex in $\St_{\eps_1}^\d (t)$ (resp. $\St_{\eps_1} (t)$) with probability at least $1 - \eta$ (this is because the probability that two cuts out of finitely many being at distance smaller than $\eps_1$ of each other is very small).
\item By the comparison results recalled at the end of Subsection \ref {USTbackground}, the law of the weighted spanning forest in $\St^\d (t)$ when restricted to the edges in $\St^\d_{\eps_1} (t)$ is dominated by the law of the weighted spanning forest in $\St^\d_{\eps_1}(t)$, and the law of the 
weighted spanning forest in $\St (t)$ when restricted to the edge in $\St_{\eps_1} (t)$ is dominated by the law of the weighted spanning forest in $\St_{\eps_1} (t)$. 
In particular, if we are given $n$ sites $s_1, \ldots, s_n$ and see that the tree in the weighted spanning forest in $\St^\d (t)$ that joins these $n$ points stays in the graph 
$\St^\d_{\eps_1}(t)$ with probability at least $A$, then this means that one can couple the weighted spanning forest in $\St^\d (t)$ and $\St^\d_{\eps_1} (t)$ in such a way that these two subtrees 
coincide with probability at least $A$ (and the similar statement holds without the superscript $\d$).
\item 
Finally, from Corollary \ref {thmm:cv-struct-graph}, we know that for all but countably many $\eps_1$, the law of the weighted spanning forest on $\St^\d_{\eps_1} (t)$ converges to that of the weighted spanning forest on $\St_{\eps_1} (t)$ as $\delta \to 0$.
\end {enumerate}
Recall that $(\tilde \St(s))_{s \in [t, 0]}$ is reconstructed from $\St(t)$ by sampling a weighted spanning forest on $\St(t)$, i.e., the limit of a weighted spanning forest in $\St_\eps (t)$ as $\eps \to 0$. 
On the other hand, $(\St (s))_{s \in [t, 0]}$ is reconstructed by taking the limit when $\d \to 0$ of the  weighted spanning forest on $\St^\d (t)$ (indeed, one reconstructs first $\St^\d (s)$ and then takes the limit $\d \to 0$). 

Combining (1) and (2) shows that for all $\eps$, one can find $\eps_1$ small enough such that for all given $\delta$, the subgraphs of $\T (\St(t))$ and of $\T (\St^\d(t))$ that join all the sites of $\St_\eps(t)$ and $\St_\eps^\d (t)$ stay respectively in $\St_{\eps_1}(t)$ and $\St_{\eps_1}^\d (t)$ with probability greater than $1- 2 \eta$. 
By (3), we see that it is therefore possible to couple these subgraphs with those obtained when sampling $ \T( \St_{\eps_1}(t))$ and $\T (\St_{\eps_1}^\d (t))$ instead of 
$\T (\St(t))$ and of $\T (\St^\d(t))$ so that they actually coincide with probability greater than $1-2\eta$. But by (4), we know that for all $\d$ small enough, these two samples can be coupled to be very close. Hence the limit (as $\d \to 0$) of the  weighted spanning forest on $\St^\d (t)$ coincides with the weighted spanning forest on $\St(t)$, which concludes the proof. Note that the argument also shows that the free spanning forest $\T(\St(t))$ is a.s. connected, hence a tree.

Mind that the identity in law between the two processes means the identity in law of all finite-dimensional marginals. And for any $t<s_1 < \ldots < s_n < 0$, we can always choose 
all the $\eps$'s and $\eps_1$'s in the above argument among those for which the convergence in Corollary \ref {thmm:cv-struct-graph} holds for these times $t, s_1, \ldots, s_n$.

\subsection{Whole plane dynamics and its properties}
Let us first observe that the previous Markov chain on structure graphs was not time-homogeneous. It was defined for all $t<0$, on the time-horizon $[t,0]$ (i.e., for a time $|t|$) as follows. First sample the UST on the structure graph, and then open each edge $e$ of this UST independently, at a  uniformly chosen time $\tau(e)$ in $[t,0]$ independently.

However, it is trivial to turn this into a time-homogeneous Markov chain. One just needs to replace the uniformly chosen times in $[t,0]$ by (positive) exponential random variables 
$\xi(e)$ with mean $1$ (one exponential variable for each edge of the structure graph), i.e., we do the time change $\xi (e) = \log (t / \tau(e) )$. Then, the edge $e$ opens at time $\xi (e)$ and one collapses it to form a new structure graph. 
As we shall now try to point out, this homogeneous-time Markov chain for structure graphs set-up turns out to be particularly interesting in the whole-plane setting. 

 Let us summarize the construction of the cutting dynamics $(\G(t))_{t\le 0}$ in the plane. Sample a continuous UST in the entire plane, and just as in the finite-volume case, define a Poisson point process on its branches, with intensity $\ell \times \mu$ where $\ell$ is the Lebesgue measure on $(-\infty, 0]$ and $\mu$ is the $5/4$-dimensional Minkowski content measure. Then, for each $t < 0$, one can cut the UST on these marked points as before, which gives rise to a collection of trees $\G(t)$, and these trees are the limit when $\d \to 0$ of their discrete counterparts $\G^\d (t)$.

Note that the processes $(\G(t))_{t \le 0}$ and $(\St (t))_{t \le 0}$ are
 scale-invariant in the following sense.
 For each $\lambda > 0$, let us define $U_\lambda (\G (t)) $ to be the forest obtained from $\G (t)$ by magnifying space by a factor $\lambda$, and $U_\lambda (\St (t))$ be the structure graph of $U_\lambda (\G (t))$, or equivalently, the graph obtained from $\St (t)$ by multiplying the edge-weights by a factor $\lambda^{5/4}$.
Then, the process $( U_\lambda (\G (t)))_{t \le 0}$ is identical in distribution to the process $(\G(t/\lambda^{5/4}))_{t \le 0}$ (and the same goes for $(\St(t))_{t \le 0}$): one can check that on the one hand, the time $0$ distributions coincides by the scale-invariance of the whole-plane UST, and on the other hand, the cutting points in the dynamics are sampled in the same way in either case, with the rescaling of time exactly corresponding to the rescaling of the Minkowski content.

 Let us now define $\pi$ to be the distribution of $\St (-1)$. Theorem \ref {main1} then states exactly that the  process $(\St ( -e^{-u} ))_{u \ge 0}$ is obtained by letting the (time-homogeneous) Markov dynamics run 
 from $\St (-1)$. But by the scale-invariance property, we get that (modulo relabeling of the edges of the structure graph), the distribution $\pi$ is invariant under the time-homogeneous Markovian dynamic. 
 
Finally, we can also note that if we start from the graph $S_0=\Z^2$ with all edge-weights equal to $1$ (or any other regular planar lattice) and let the time-homogeneous Markov chain $(S_u)_{u \ge 0}$ run until a large time $U$, we discover each edge of the final UST on $\Z^2$ (independently) with probability $1- e^{-U}$ (or more exactly, rather than their edges, their ``traces on the structure graphs''). In particular, with Theorem \ref{main1}, this shows that (modulo relabeling of the edges of the structure 
graph, i.e., scaling down $\Z^2$ to $\d \Z^2$ for an appropriately chosen $\d$ depending on $U$), as $U \to \infty$, the law of the structure graph converges to $\pi$ (in the sense of Corollary \ref{thmm:cv-struct-graph}). 

Hence, this provides the following renormalization flow description of the UST scaling limit via (a rescaling of) the time-homogeneous Markov chain $P_u$ on the state of discrete weighted graphs: 

\begin {thmm}[Renormalization flow description] 
\label {main2}
The measure $\pi$ (that describes the previous scaling limit of near critical spanning forests) is invariant under the Markov chain. 
Furthermore, the (time-homogeneous) Markov chain started from any deterministic periodic two-dimensional transitive lattice and properly rescaled converges in distribution to $\pi$. 
\end {thmm}

\section {Technical estimates and proofs} 
\label {lastsection}

\subsection {First comments about the structure graphs and their convergence}
Most of the remainder of this paper is now devoted to the proof of Proposition \ref{lem:weightSt}, which provides the convergence of the discrete structure graph weights to 
their continuous counterparts. 
In this section, we are working with the UST on the whole plane but the proofs can easily be extended to any bounded domain with $\C^1$ boundary.

Let us now make some comments about this, and explain how to deduce Proposition \ref{lem:weightSt} from two lemmas (Lemma \ref {finitelength} and Lemma \ref {prop_lengthdiscret}) that we will then prove in the 
subsequent section, based on more ``traditional'' arm-estimates and considerations for UST.

Suppose first that $z_0$ and $z_1$ are two given points. In both the discrete and continuous settings, these two points are joined by a unique path in the UST, which has a finite (renormalized) length (or Minkowski content --- by slight abuse of terminology, we will now use the word length also in the continuous case), so that the number of ``cuts'' on this branch (conditional on this length, and for a given $t$) follows a Poisson distribution. If these two points $z_0$ and $z_1$ end up in different trees at the end of the cutting procedure, then there has been a ``first cut'', i.e., an edge $e$ on this path that has been removed first (when one looks back from time $0$ to time $t$ in the cutting procedure), and its law (conditional on the branch between $z_0$ and $z_1$) is uniform on this branch with respect to length. Mind that the edge $e$ has a positive probability not to exist (if there was no cut on the branch). 

If we consider the entire UST and remove from it just this one edge $e$, then one has divided the UST into two trees, one containing $z_0$ and the other one containing $z_1$. As the graph dual to the whole-plane UST is also a UST, the intersection between the boundaries of the two trees containing respectively $z_0$ and $z_1$ is a cycle ${\mathcal C}^\d$, which consists of the edge $e^\star$ dual to $e$ together with the branch in the dual of the UST that joins the two extremities of $e^\star$. 
Clearly, if one removes more edges than just $e$, the trees that contain $z_0$ and $z_1$ respectively will shrink, and the intersection between the boundary of these two trees can only decrease. 
Hence, the interface between the two clusters of $\G^\d (t)$ that contain $z_0$ and $z_1$ is a subset of this cycle (and its length is bounded 
by that of ${\mathcal C}^\d$). The same situation occurs in the continuous case. Here, when one chooses a first point $z$ at random (according to Minkowski-content) on the UST branch joining 
$z_0$ and $z_1$, one can consider the cycle ${\mathcal C}$ in the dual tree that joins $z$ to itself, and when one removes more points according to the cutting dynamics, the clusters that contain the two points $z_0$ and $z_1$ will intersect along a subset of that cycle ${\mathcal C}$. 
\begin{figure}[ht!]
\centering
\includegraphics[width=7cm]{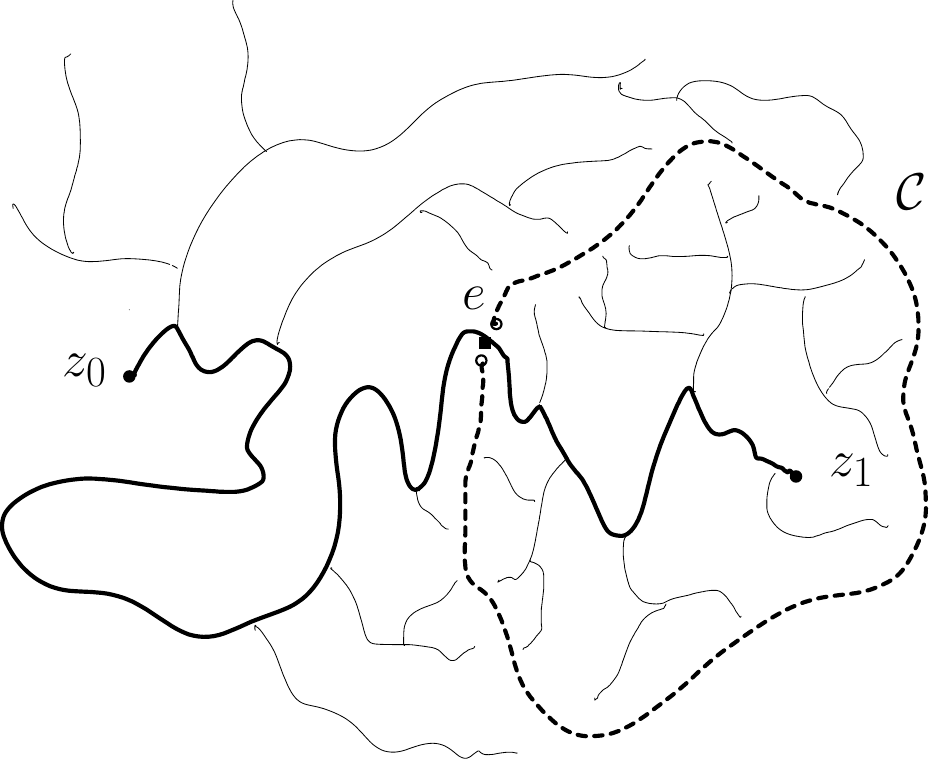}
\hskip 1cm
\includegraphics[width=7cm]{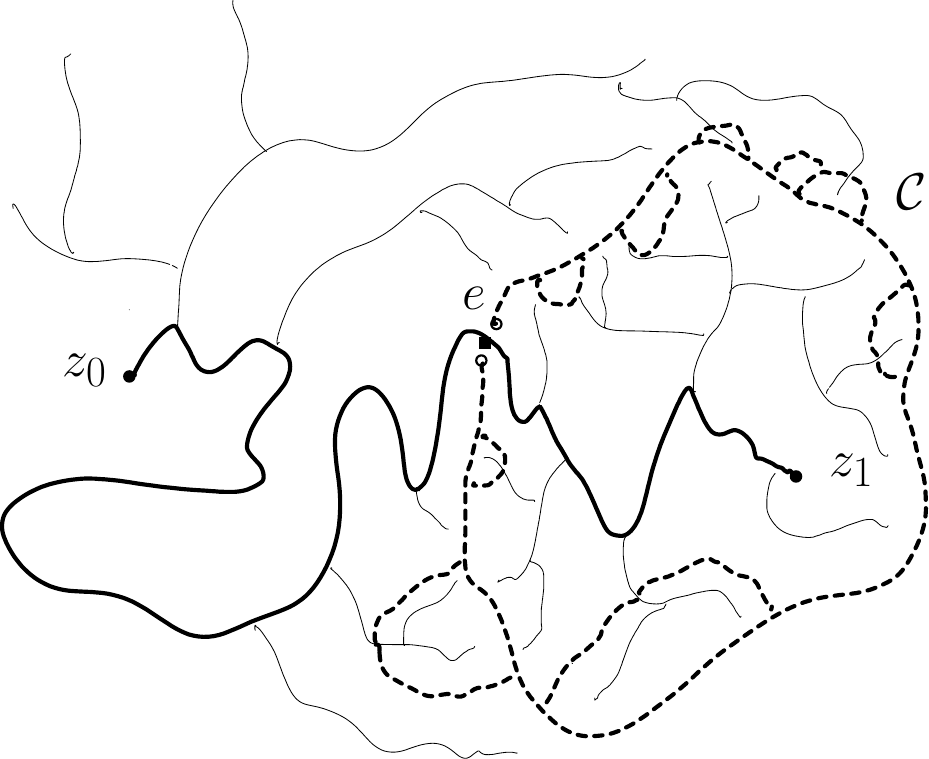}
\caption{Sketch of the tree, of the cycle $\mathcal C$ and the cuts}
\end{figure}

Let us first state a simple consequence of the convergence in distribution of ${\mathcal C}^\d$ combined with the convergence of the renormalized length 
on the branch from $z_0$ to $z_1$. In the following, $B(z,r)$ denotes the Euclidean ball of radius $r$ centered in $z$ when $z \in \CC$ and when $e$ is an edge of $\CC^\d$, $B(e,r)$ is the ball of radius $r$ centered at its midpoint.

\begin {lem}
\label {l1}
As $\eta_0 \to 0$, the probability that $\C^\d \not\subset B(0, 1/\eta_0)$ or $d ( z_0, e) < \eta_0$ or $d(z_1, e) < \eta_0 $ occurs goes to $0$ uniformly with respect to $\d$.  
\end {lem}

\begin{proof}Consider $\d_k \to 0$. 
By contradiction, if for all $\eta_0$, $\C^{\d_k} \not\subset B(0, 1/\eta_0)$ occurs with uniformly positive probability for infinitely many $\delta_k$, then this  
would readily imply that the continuous whole-plane UST is disconnected, whereas $d ( z_0, e) < \eta_0$ (or $d(z_1, e) < \eta_0 $) occurring with uniformly positive probability
for infinitely many $\d_k$ would contradict the finiteness of the Minkowski content of the branch from $z_0$ to $z_1$ in the continuous whole-plane UST.
\end{proof}

We know already that the lengths of branches in the dual tree that join prescribed given points do converge to their continuous counterparts, but
care will be needed when we want to show the convergence of the length of the entire cycle ${\mathcal C}^\d$, because it does originate at a special point, i.e., 
a point on the backbone of the original UST, so we need to exclude the scenario where something weird happens to the length of ${\mathcal C}^\d$ in the vicinity of this special point. 
This is the purpose of the next lemma: 

\begin{lem}\label{finitelength}
Let us fix $\eta_0, z_0$ and $z_1$, and condition on the event that $\C^\d$ exists, and that the three events in Lemma \ref {l1} do not occur (note that this is a conditioning on 
an event of positive probability, bounded from below independently of $\d$, and that then, the diameter of ${\mathcal C}^\d$ is bounded from below by $\eta_0$). 
As $\eta$ goes to $0$, in the previous setting (for fixed $z_0$ and $z_1$), the expected (conditional) renormalized length  $u^\d (\eta)$ of the 
 intersection of ${\mathcal C}^\d$ with the ball of radius $\eta$ around the center of $e^\star$, does tend to $0$ uniformly with respect to $\d$. 
\end{lem}

Next, one can make the following observations (which can be made rigorous, but they serve here as a motivation and will not be used later, so we will not bother to do so). Suppose that in the previous scenario, one considers the continuous tree containing $z_1$ after cutting away just $e$, and that this tree is bounded (if we were in the whole plane, this means that $z_1$ was on the bounded side of the cut $e$).
Lemma \ref {finitelength} indicates that the length of $\mathcal C$ (in terms of Minkowski content) is finite. 
However, we need to understand something finer, namely what the common boundary of the sub-trees containing $z_0$ and $z_1$ looks like at time $t$ of the cutting procedure, when one has removed from $\G^\d(0)$ many more edges than just $e$. One can notice that for a ``typical point'' on the cycle $\mathcal C$, a similar argument will show that the (Minkowski-content) length between this point $z$ and $z_1$ in the initial tree is finite. 
Hence, this point will have a positive probability to be cut off from $z_1$, but it also has a positive probability not to be cut off. Hence, the expected portion of the length of the part of $\mathcal C$ that will remain on the 
outer boundary of the cluster containing $z_1$ is in fact positive.  
On the other hand, a back-of-the envelope calculation (that we do not reproduce here) suggests that the total length of the tree consisting of all the branches that join $z_1$ to all the boundary points in $\mathcal C$ is infinite. This means that an infinite number of portions of $\mathcal C$ will be cut out. In other words, the situation is that one starts with $\mathcal C$ and cuts off infinitely many connected arcs from it, and these arcs will 
be dense on $\mathcal C$, but the total length of the remaining set can still be positive. 

The purpose of the following lemma is now to 
control this feature at the discrete level:
Let us say that a point $z$ of ${\mathcal C}^\d$ is cut-out from this boundary at a scale smaller than $\e$ if there exists a cut disconnecting $z$
from one of the two extremities of the special edge $e$, in such a way that the part of the tree disconnected from $e$ by this cut has a diameter smaller than $\e$.
For each $\eta >0$, we are going to define $L^\d (\eps)$ to be the renormalized length of the set of points on ${\mathcal C}^\d \cap ( B(0, 1/ \eta) \setminus B(e, \eta))$ that are cut-out from the interface ${\mathcal C}^\d$ at a scale smaller than $\e$:

\begin{lem}\label{prop_lengthdiscret}
As $\e$ goes to $0$, in the previous setting (for fixed $z_0$, $z_1$ and $\eta$), the expected value of $L^\d (\e)$ tends to $0$ uniformly with respect to $\d$.
\end{lem}

\begin{figure}[ht!]
\centering
\includegraphics[width=7cm]{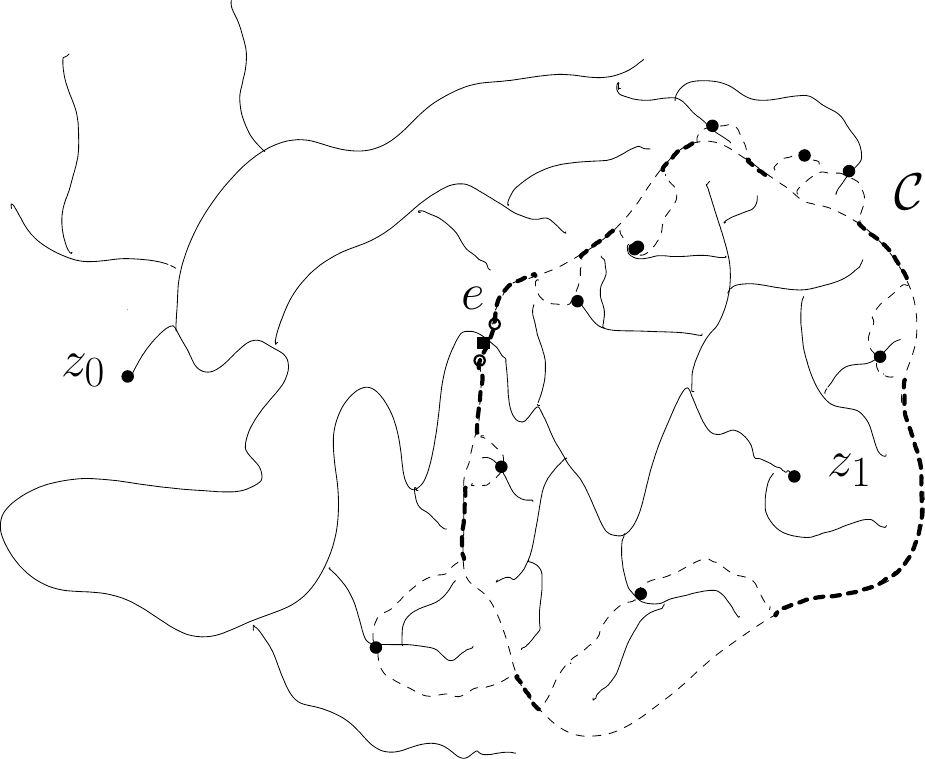}
\caption{After all the cuts: the remaining interface between the trees containing $z_0$ and $z_1$}
\end{figure}

We shall prove Lemmas \ref{finitelength} and \ref{prop_lengthdiscret} in the next section, but let us already explain now how Proposition \ref{lem:weightSt} follows from them: 

\begin {proof}[Proof of Proposition \ref {lem:weightSt}]
Consider points $z_0,z_1$ as well as a sequence of mesh sizes $\delta_k \to 0$.  
We can sample the graphs $\G^{\d_k}(t)$ for all $k$, together with $\G(t)$ on the same probability space, in such a way that the tree containing (a $\d_k$-approximation of) $z_0$ (resp. $z_1$) in $\G^{\d_k}(t)$ converges almost surely to the tree containing $z_0$ (resp. $z_1$) in $\G(t)$. We can furthermore choose our setup so that the dual tree at time $0$ converges almost surely, in the sense that the renormalized lengths of branches of its finite subtrees do (Corollary \ref{prop:joint}). We will spend the remainder of this proof showing that $l^{\d_k}(z_0,z_1)$ converges in probability to $l(z_0,z_1)$ as $k \to \infty$.
More precisely, we need to show (for each $\eta_0$) this convergence on the event described in Lemma \ref{l1}, i.e., when the cycle separating $z_0$ from $z_1$ is not very large, and when the cut edge $e$ is neither very 
close to $z_0$ nor very close to $z_1$. 

Let us first note that for the coupling of the trees and of the cutting processes introduced in the proof of Proposition \ref{propo:CV_cutoff_process}, the cycles ${\mathcal C}^{\d_k}$ are with very high probability 
very close to ${\mathcal C}$. We can therefore actually choose such a coupling and assume that almost surely, ${\mathcal C}^{\d_k}$ do converge to ${\mathcal C}$ as planar curves. 

The next step is to prove convergence of the lengths of these cycles, which is where Lemma \ref{finitelength} is crucial. It ensures that the lengths of the two portions
of ${\mathcal C}^{\d_k}$ near the special edge $e$ (from $e$ to the  circle of radius $\eta$ around the center of $e$) tends to $0$ uniformly in $\delta$, when $\eta \to 0$. 
On the other hand, Schramm's strong approximation result (Result \ref{th.finiteness.S2}) for the dual tree, together with the strong convergence of finite subtrees of the dual tree (as shown in the appendix) shows that the bulk lengths (i.e. that the lengths of the portions of ${\mathcal C}^{\d_k}$ obtained by removing its portions near the special edge $e$) do converge (indeed, the strong approximation result implies with a probability as close to one as one wishes,  
all the pieces of the dual tree that are not near to its leaves will be contained in some finite subtree with prescribed endpoints (one chooses enough of these endpoints deterministically so that the probability gets close to $1$), and we know that the convergence of the length-parametrization for this finite subtree holds). 

We now need to control the length of the discrete approximations of $l(z_0, z_1)$. 
We now choose $\e$, and denote by $l_\e^\d$ (resp. $l_\e$)  the renormalized length of the set of points in $\C^{\d}$ (resp. $\C$)  that have not been disconnected from $z_0$ or $z_1$ at a scale larger than $\e$ (i.e., by a cut creating a cycle of diameter larger than $\e$). In other words, we remove from the total length of $\C^{\d}$ (resp. $\C$) the contribution of all the macroscopic cuts (of diameter larger than $\e$). 

By definition of $l(z_0, z_1)$, we know that $l_\e \to l(z_0, z_1)$ almost surely as $\e \to 0$. Moreover, Lemma~\ref {prop_lengthdiscret} ensures  that 
$\E(l_{\e}^{\d_k} - l^{\d_k} (z_0, z_1)) $ goes to $0$ as $\e \to 0$, uniformly in $k$. As a consequence, we get that for all $r>0$, one can find $\e_0$ such that for all $\eps \le \eps_0$, there exists $k_0$ so that for all 
$k \ge k_0$, all the probabilities 
 $\P(|l_{\e} - l(z_0, z_1)|>r)$, $\P(|l_{2{\e}}^{\d_k} - l^{\d_k} (z_0, z_1)|>r)$  and $\P(|l_{{\e}/2}^{\d_k} - l^{\d_k} (z_0, z_1)|>r)$ are smaller than $r/10$.
We then choose such an $\e_0$ and $k_0= k_0 (\e_0)$. 
We know that the finitely many pieces of $\C^\d$ cut by cycles of diameter larger than ${\e_0}$ do converge almost surely to their continuous counterpart (for the same reason that the curve $\C^\d$ converges to $\C$). In particular, this shows that one can find $k_1 \ge k_0$ so that for all $k \ge k_1$, the probability that 
\[
   l_{2 {\e_0}}^{\d_k} + (r/2) \ge l_{\e_0} \ge  l_{{\e_0}/2}^{\d_k} - (r/2)
\]
is greater than $1 - r/10$. 
Hence, for $k \ge k_1$, with probability at least $1- r /2$, the four quantities $l_{2{\e_0}}^{\d_k}$, $l_{{\e_0}/2}^{\d_k}$, $l^{\d_k} (z_0, z_1)$ and $l_{\e_0}$ are no more than $3r$ apart.

Wrapping up, we see that for any fixed $r$, one can find $k_1$, so that for all $k \ge k_1$, 
$$\P(|l^{\d_k} (z_0, z_1)-l(z_0, z_1)|>4r) \le r.$$ 
In other words, $l^{\d_k} (z_0, z_1)$ converges in probability to $l (z_0, z_1)$ as $k\to \infty$.
\end {proof}

\subsection {Arm events in UST} 

Let us first recall an estimate about LERW of the type that is essential in the derivation of results involving the Minkowski-content in \cite {AKM,BVL,LV}.
Let $X$ and $Y$ be two independent simple random walks on $\Z^2$ starting at $x$ and $0$ respectively and stopped at their first exit time $\tau_X$ and $\tau_Y$ of the ball of radius $N$ around the origin. 
Let us consider the loop erasure $\hat Y$ of $Y$.
Take $L < N$ and denote by $\hat Y^L$ the subpath of $\hat Y$ from its last hitting time of the ball of radius $L$ around the origin until its end $\tau_Y$, and define the escape probability to be
\[
\Es (L,N) := \P_{x=0} ( X \cap \hat Y^L = \emptyset ).
\]
\begin{res}\label{ELN2}
There exists a constant $C >0$ such that for all $L$ and $N$ with $L \leq N/2$,
 \begin{align*}
C^{-1} (L/N)^{3/4} \leq  \Es (L,N)   \leq  C (L/N)^{3/4}
 \end{align*}
\end{res}

\begin {proof}
When $L=1$, the estimate can be derived following the proof of \cite[Corollary 3.15]{Barlow}, using the better estimate of \cite[Theorem 1.1]{BVL} as an input.
Moreover, one can compare $\Es(L,N)$ to $\Es(1,N)/\Es(1,L)$ thanks to \cite[Propositions 5.2 and 5.3]{Masson}, which proves Result \ref{ELN2}.
\end{proof}

Note that this implies that the probabilities, say, $\Es (L,N)$ and $\Es (5L, N)$ are comparable.

Moreover, the case  $L=1$ provides (via Wilson's algorithm) the probability that two distinct branches in the wired UST in $B(N)$ that start at the origin and  next to the origin stay disjoint until they hit the circle of radius $N$. By duality, this is also (almost, as there is the issue of the $(1/2, 1/2)$ translation) the probability that for the free UST in $B(N)$, there exists a branch from the boundary $\partial B(N)$ to itself that goes through a fixed edge neighboring the origin.

\medbreak

An event related to the previous non-intersection events, but slightly different 
is the following arms event $\mathcal{A} (L,N)$ around the origin between scales $L$ and $N$ that there exists four disjoint branches, two of the UST and two of the dual UST 
(in alternating order)
that connect  $\partial B (L)$ to $\partial B (N)$ (see Fig. \ref{armsannulus}). 
\begin{figure}[!ht]
\centering
\includegraphics[width=7cm]{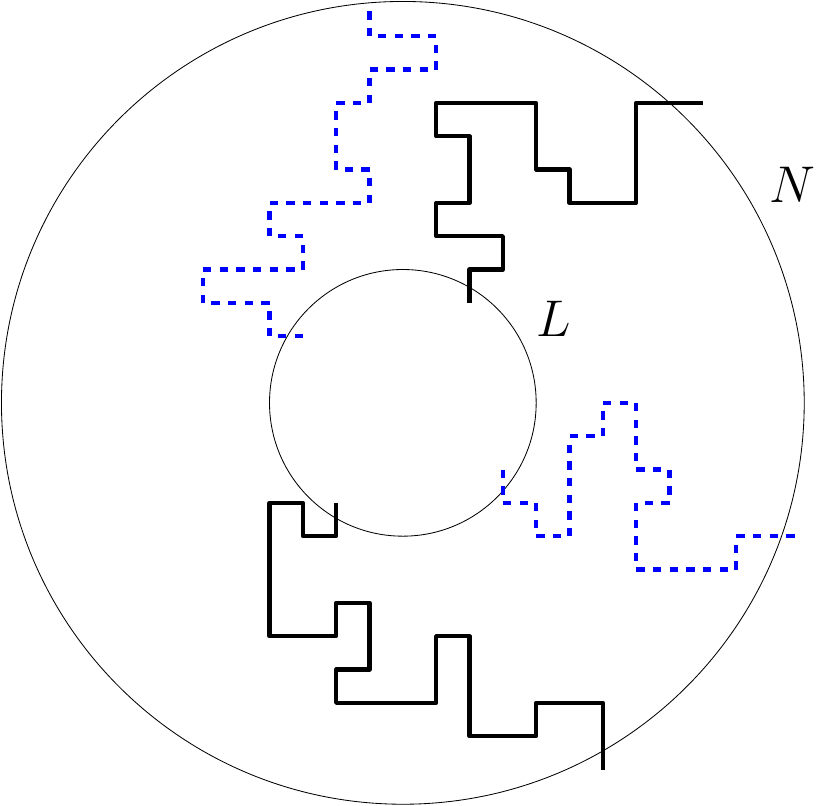}
\caption{The four alternating disjoint branches in the UST.}\label{armsannulus}
\end{figure}

\begin{lem}\label{ALN}
Consider the UST in a discrete domain $\Omega \subseteq \Z^2$ containing $B(N)$, with arbitrary boundary conditions. There exists a constant $C >0$ independent of $\Omega$, $N$  and the boundary conditions, such that for all $L \leq N$,
\begin{align*}
\P(\mathcal{A} (L,N)) \leq C (L/N)^{3/4}.
\end{align*}
\end{lem}

\begin{proof}[Proof of Lemma \ref{ALN}]
It is clearly sufficient to focus on the case where $L < N /10$ say (choosing also $C \ge 10$ at the end will then ensure that the statement holds for all $L \le N$). 
Note that (for whatever $\Omega$ and boundary conditions), when ${\mathcal A}(L,N)$ occurs, then at least one of the following two events occur:
\begin {itemize}
\item There exists a branch of the UST that crosses the annulus twice: a part of this branch starts from $\partial B(N)$, reaches $\partial B(L)$ and then hits 
$\partial B(N)$ again.
\item There exists a branch of the dual UST that crosses the annulus twice: a part of this branch starts from $\partial B(N)$, reaches $\partial B(L)$ and then hits 
$\partial B(N)$ again.
\end {itemize}
By duality, it is enough to evaluate the probability of the first event, which we call ${\mathcal A}'(L,N)$. We then can note that the monotone coupling of USTs with different boundary conditions and in different domains shows that the probability of ${\mathcal A}'(L,N)$ is maximal (among all domains and boundary conditions, but for fixed $L$ and $N$) for the 
ball $B(N)$ with free boundary conditions on $\partial B(N)$. 
By then considering its dual configuration again, this means that it is sufficient to bound the probability that, for a UST in $B(N)$ with wired boundary conditions, there exists two disjoint branches of the tree that join  $\partial B(L)$ to the outer wired boundary $\partial B(N)$. 

From now on in this proof, we will work with the UST in $B(N)$ with wired boundary conditions. Since the branch $\g$ of this UST from the origin to $\partial B(N)$ always joins $\partial B(L)$ to $\partial B(N)$, we want to show that the probability that there exists at least another branch (disjoint from $\g$) that joins $\partial B(L)$ to $\partial B(N)$ is bounded by a constant times $(L/N)^{3/4}$.  
To see this, we consider the UST, conditionally on $\g$ and we start constructing the rest of the tree by a variant of Wilson's algorithm that we now describe.

Let $x$ denote the first (when starting from the origin) point on $\g$ that is at distance greater than $2L$ from the origin. 
The first step of our iteration goes as follows. We will use a random walk $X^1=X$ starting from this point $x$. 
More specifically, we start Wilson's algorithm at the the first point $X(j_1)$ at which $X$ is not in $\g$, and we use the movements of $X$ to perform it. 
In this way, one attaches a branch $I_1$ from $X(j_1)$ to $X(j_1') \in  \g \cup \partial B(N)$ (with obvious notations: here $j_1'$ is the first time after $j_1$ at which 
$X(j_1') \in  \g \cup \partial B(N)$) that will be part of the UST. Then, one continues still using $X$: the next point where one will start Wilson's algorithm will correspond to the first time $j_2$ larger than $j_1'$ that is not on $\g \cup \partial B(N) \cup I_1$. 
One continues like this constructing branches of the uniform spanning tree using this random walk 
$X$, until the first time greater than the exit time of $B(3L)$ by $X$ at which $X$ hits $\gamma \cup \partial B(N)$. 
At that moment, one has constructed some subtree of the wired UST that consists of the union of $\gamma$ with a finite number of branches $I_1, \ldots , I_k$.  

The key observation can be loosely speaking described as follows: during this  first step, one can create at most one second branch of the UST that crosses ${\mathcal A}(L,N)$ (and the probability of this event will be bounded by some constant times $\Es (L, N)$), and on the other hand, with 
some positive probability, one has drawn a collection of branches that do actually prevent the existence of a second branch of the UST that crosses ${\mathcal A}(L, N)$. Let us be more specific:  
let $E_1$ be the event that:  

\begin {itemize}
 \item 
The walk $X$ first winds once around the origin in the annulus $B(3L) \setminus B(L)$ -- in other words the first time $\rho_1$ at which the argument of 
$X$ around the origin exits $[-2 \pi + \arg (x), 2 \pi + \arg (x)]$ is smaller than the exit time $\rho_2$ of $B(3L) \setminus B(L)$ by $X$. 
\item And then, after $\rho_1$ but before exiting the ball of radius $L$ around $x$, the walk $X$ makes a closed loop around $x$ within $B(x, L/2)$ (i.e., it contains a path that disconnects $x$ from $\partial B(x, L/2)$). 
\end {itemize}

One can note that when this event $E_1$ occurs then necessarily, 
there cannot exist a second branch of the wired UST (disjoint from $\gamma$) that connects $\partial B(L)$ to $\partial B(N)$. Indeed, after $\rho_1$, $X$ will 
touch the branch $\gamma$ in $B(x,L/2)$ at at least one point $y$ such that the part of $\gamma$ joining $x$ and $y$ stays in $B(x,L/2)$: let us call $\tau$ such a time. By our definition of $E_1$, we see that the 
winding of $X$ around the origin between $0$ and $\tau$ will differ from that of the part of $\gamma$ that joins $x$ to $y$. We now call $\tau'$ to be the first time at which $X$ is at some point $y'$ on $\gamma$ so that 
the winding of $X$ until time $\tau'$ differs from that of the part of $\gamma$ that joins $x$ to $y'$. Note that on the event $E_1$,  $\tau' \le \tau \le \rho_2$, and 
that the time $\tau'$ corresponds to some moment in our algorithm, where one has constructed branches 
$I_1, \ldots, I_{k'}$ for some $k' \le k$. 
It is then easy to see that the set of vertices in $\gamma \cup I_1 \cup  \ldots \cup I_{k'}$ disconnects $B(L)$ from $\partial B(N)$, which prevents the existence of any UST branch disjoint from $\gamma$
that connects $\partial B(L)$ to $\partial B(N)$. 

Now, note that the probability of $E_1$ is bounded from below by a universal constant $b$, independent of $L$, $N$ or $\g$.
Indeed, the probability of $E_1$ converges to that of the corresponding event for a Brownian motion in an annulus as $L \to \infty$.

If $E_1$ does not hold, then we continue our algorithm until the first time after $X$ exits $B(3L)$ at which $X$ hits $\g$ or $\partial B(N)$. 
The probability of hitting $\partial B(N)$ is bounded from above by a constant times the conditional probability (given $\g$) that a random walk started from $0$ does not hit 
the subpath $\gamma^{5L}$ of $\gamma$ -- let us call this probability $p(\g)$: this is because the exit measures on $\partial B(5L)$ of random walks started at points inside of $B(3L)$ are all absolutely continuous with respect to another, with Radon-Nikodym derivatives uniformly bounded (with respect to $L$ and to the starting points of the walks). Then, we simply iterate the same procedure, starting an independent random walk $X^2$ from $x$ again, except that we already have added some branches to the UST, so that the way we add branches to the tree using Wilson's algorithm is slightly modified. We can however use the event $E_2$ defined for the random walk $X^2$ in the same way as $E_1$ was defined for 
$X^1$. We then iterate the procedure.

Then, we obtain the following iterative scheme. We first discover $\g$ and then: 
\begin {itemize}
\item With a probability at least $b$, the event $E_1$ occurs, and we then know that there is no second branch of the UST joining $\partial B(L)$ and $\partial B(N)$.
\item If not, then, with a conditional probability bounded from above by $p(\g)$, one discovers a second branch of the UST joining $\partial B(L)$ and $\partial B(N)$. 
\item Then, with a conditional probability at least $b$ again, the event $E_2$ occurs, and we know that there is no further branch of the UST joining $\partial B(L)$ and $\partial B(N)$.
\item If $E_2$ does not occur, then with a conditional probability bounded by $p (\g)$, one discovers an extra branch of the UST joining $\partial B(L)$ and $\partial B(N)$.
\item And so on.
\end {itemize}
Hence, we see that conditionally on $\g$, we can bound the expectation of the number $\mathcal{N}$ of additional disjoint branches (apart from $\g$) in the wired UST that join $\partial B(L)$ to $\partial B(N)$:
\[ \E(\mathcal{N}|\gamma)  \leq C\ \sum_{ k \ge 1 } (1-b)^k p(\g) =  C\ \frac {1-b}{b} \times p(\g).\]
Since 
\[\E ( p (\g) ) = \Es (5L,N),\]
we conclude that 
\[ 
 \ \P(\mathcal{A}'(L,N)) \le  \ \E(\mathcal{N})  \le \frac {1-b}{b}  \times C' (L/N)^{3/4}.
\]
\end{proof}

\subsection {Arm-estimates imply Lemma \ref{prop_lengthdiscret} and Lemma \ref {finitelength}}

\begin{proof}[Proof of Lemma \ref{prop_lengthdiscret}]
We can bound the expected value of the renormalized length $L^\d (\eps)$ by $\d^{5/4}$ times the sum over all pairs of edges $e_0$ (in the dual lattice) and $e_1$ (in the original lattice) that are at distance at most $\eps$ of each other of the probability of the intersection $E(e_0,e_1)$ of the following events : 
\begin {itemize}
 \item The edge $e_0$ belongs to the dual cycle ${\mathcal C}^\d$ (that appears when closing the edge $e$).
 \item The edge $e_0$ is at distance greater than $\eta$ from $e$, and in the ball of radius $1/ \eta$ around the origin.
 \item If we erase the two edges $e$ and $e_1$ from the UST, the edge $e_0$ is no longer on the interface between the clusters that contain $z_0$ and $z_1$.
 \item The edge $e_1$ is cut out during the cutting procedure (note that this event occurs independently of the rest, with probability $\d^{5/4}$ times a constant that depends on $t$). 
\end {itemize}
\begin{figure}[!ht]
\centering
\includegraphics[width=9.5cm]{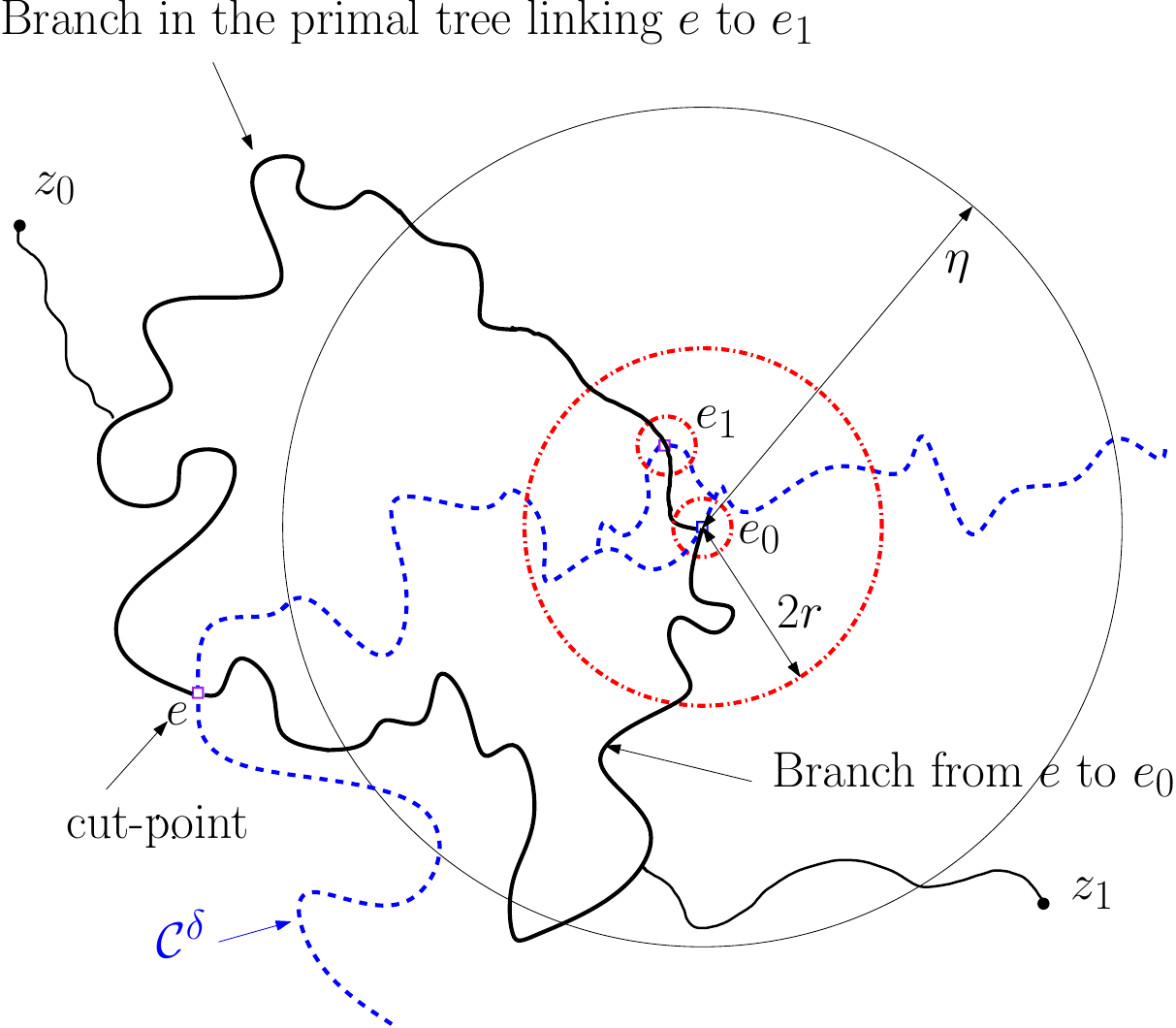}
\caption{Arm events appearing in $E(e_0,e_1)$}\label{lemme-epscut}
\end{figure}
For any  edge $f$ in $\d \Z^2$ and $l_1 \leq l_2$, denote by ${\mathcal A}_{f} (l_1,l_2)$ the four arms event in the annulus $B(x,l_2) \setminus B(x,l_1)$ centered at the middle point $x$ of the edge $f$ for the UST on $\d \Z^2$.
We have that, if we set $r := d(e_0,e_1) $, then (see Fig. \ref{lemme-epscut}):
\[
E ( e_0, e_1) \subset {\mathcal A}_{e_0} (\d/2,r/3) \cap {\mathcal A}_{e_1} (\d/2,r/3) \cap {\mathcal A}_{e_0} ( 2r, \eta ).
\]

Using the upper bounds on the probabilities of these events given by Lemma \ref{ALN}, together with the fact that the bounds on the first two are independent of the boundary conditions (so it is possible to first condition on the last one, and then to bound the conditional probability of the first two), we get readily that
\begin {align*}
  \E(L^\d (\eps)) &  \le \delta^{5/4} \times C(t) \delta^{5/4} \times \sum_{e_0 \in B(0, 1/\eta)}  \sum_{e_1 \in B(e_0, \eps)} (r/\delta)^{-3/4} \times (r / \delta )^{-3/4} \times (\eta / r)^{-3/4}
  \\
  & \le C(\eta) \delta^{5/4 + 5/4 + 3/4 + 3/4 - 2  } \times \sum_{x \in \delta \Z^2 \cap B(0, \eps) \setminus \{0\}} |x|^{-3/4}  
  \le C(\eta) \eps^{5/4}. 
  \end {align*}
\end{proof}

\begin {proof}[Proof of Lemma \ref{finitelength}]
The proof goes along similar lines than the previous one. We can bound the expected renormalized length by $\delta^{5/4}$ times the sum over all pairs of edges $e_0$ (in the original lattice) and $e_1$ (in the dual lattice) such that $r:= d(e_0, e_1) \le \eta$ of the probability of the intersection of the following events: 
\begin {itemize}
 \item The edge $e_0$ is on the UST branch from $z_0$ to $z_1$, and it splits this branch in two parts of diameter larger than $\eta_0$.
\item The edge $e_0$ is removed by the cutting procedure.
 \item The edge $e_1$ belongs to the dual cycle that appears when closing the edge $e_0$.
\end {itemize}
As in the previous argument, we can note that this event (for given $e_0$ and $e_1$) is included in the joint occurrence of the four-arms events ${\mathcal A}_{e_0} (\delta/2, r/3)$, $ {\mathcal A}_{e_1} (\delta/2,r/3)$ 
and ${\mathcal A}_{e_0} ( 2r, \eta_0)$, and we can conclude using the very same computation.
\end {proof}

\medbreak

\appendix

\section{Strong convergence of Uniform Spanning Trees}\label{sec:app}

In this appendix, we show that USTs with different boundary conditions (wired, free, whole-plane) converge when the mesh-size vanishes, in the sense that their finite subtrees parametrized by their 
appropriately renormalized length converge in law. 

This is done by combining three ingredients: 
the result of Lawler and Viklund about convergence of radial LERW to SLE$_2$ in its natural parametrization (Result \ref{res:convLERW}), the convergence of 
discrete USTs to their continuous counterparts (up to time-reparametrization) by \cite {LSW_LERW}, and
absolute continuity arguments between USTs with different boundary conditions using loop-soups. Discrete harmonic measure estimates will be instrumental as well. 

In order to avoid lengthy but easy details, we do outline here the main ideas, leaving the simple considerations to the interested reader.

\subsection{Notations and background}
$\Omega$ will denote a bounded simply connected domain with analytic boundary.
For each $z \in \Omega$, we let $z^\d$  be a point in $\d \Z^2$ that is at distance less than $\d$ from $z$ chosen in some deterministic way. 
We choose some fixed $\xi \in \Omega$ and we define $\Omega^\d$ to be the connected component of the graph $\delta \Z^2 \cap \Omega$ that contains $\xi^\d$ (where edges are kept in this graph if they lie entirely in $\Omega$). 
A half-edge adjacent to a vertex in $\Omega^\d$ but that does not belong to an edge in $\Omega^\d$ will be called a boundary half-edge. 

On this discretization, we will define the usual two measures on random walk paths and random walk loops. 
Each of these measures will come in two variants, corresponding the two following Markov chains: 
\begin {itemize}
 \item The usual random walk in $\d \Z^2$ killed upon exiting $\Omega^\d$ (i.e., at the first time it used an edge that is not in $\Omega^\d$). 
 \item The reflected random walk in $\d \Z^2$ that at each step, is choosing a neighbor with probability $1/4$ and jumps to it if the corresponding edge is in $\Omega^\d$. If the corresponding edge is not in $\Omega^\d$, it stays put and we say that the walk bounced on the boundary half-edge it tried to explore at that time (two walks that stay put at $x$ but bounce on different half-edges will be considered to be different in what follows). 
\end {itemize}
The measures corresponding to the random walk killed upon exiting $\Omega^\d$ are the following.
We let  $\lambda_\delta$ denote the (oriented, unrooted) loop measure in $\Omega^\delta$: an unoriented 
unrooted loop $X$ has a mass $J^{-1} \times (1/4)^{|X|}$, where $J$ is the multiplicity of the loop 
(i.e., $J$ is the maximal integer $j$ such that the loop is the concatenation of $j$ times the same loop,  see for instance \cite{W2,W3}).

We define similarly the loop measure $\lambda_\d^r$ corresponding to the reflected random walks. It is worthwhile to notice that when one restricts $\lambda^r_\d$ to the set of loops that do 
not bounce on the boundary (i.e., that do not use any boundary half-edge), one gets exactly the measure $\lambda_\d$. 

We will also use  the (oriented) excursion measure on nearest neighbor paths in $\Omega^\d$ that we denote by $\nu_{\Omega^\d}$. This is the measure that assigns a mass $4^{-n}$ to the 
nearest-neighbor paths $X= (X_0, \ldots, X_n)$ in $(\d \Z)^2$ such that the edges $(X_0, X_1)$ and $(X_{n-1}, X_n)$ are not in $\Omega^\d$, and the other $n-2$ edges are in $\Omega^\d$. 

\medbreak

Let us recall that one way to sample a UST in $\Omega^\d$ with wired boundary conditions using Wilson's algorithm (for convenience, we will view the boundary as a single vertex $\partial$) is to iteratively sample the subtrees of the UST $ \T_{\partial, z_1^\delta, \ldots, z_n^\delta}$ that join the points $z_1^\delta, \ldots, z_n^\delta$ and the boundary. The tree 
 $ \T_{\partial, z_1^\delta, \ldots, z_{n+1}^\delta}$ is obtained by adding to  $ \T_{\partial, z_1^\delta, \ldots, z_n^\delta}$ an independent LERW joining $z_{n+1}^\delta$ to  $ \T_{\partial, z_1^\delta, \ldots, z_n^\delta}$.
 The probability that $\T_{\partial, z_1^\delta, \ldots, z_n^\delta}$ is a given tree $T$ (see \cite[Chapter 2, Section 3]{W2} -- this is related to the fact that the set of loops that are being erased in Wilson's algorithm can be 
 interpreted exactly as the set of loops in a loop-soup that do intersect the UST branches that one constructs) is equal to some renormalization constant (i.e., independent of $T$) times $U_\d (T) \times 4^{-|T|}$, where $|T|$ is the number of edges in $T$, and 
\[
U_\d (T) := \exp  \left( \lambda_\d \left( \{ \ell \ , \ z_1^\delta \not\in  \ell, \  \ell \cap T \not= \emptyset \}\right)\right)
\]
(here we use the fact that the set of loops that intersect $z_1^\delta$ does not depend on $T$, and so 
we can include the term $\exp \left(\lambda_\d \{ \ell \ , \ z_1^\delta \in \ell \}\right)$ in the renormalizing constant).

Similarly, when one samples a UST in $\Omega^\d$ with free boundary conditions, the probability that $\T_{z_1^\delta, \ldots, z_n^\delta}$ is a given tree $T$ is given by $V_\delta(T) \times 4^{-|T|}$, where
\[
V_\d (T) := \exp  \left( \lambda_\d^r \left( \{ \ell \ , \ z_1^\delta \not\in  \ell, \  \ell \cap T \not= \emptyset \}\right)\right).
\]
(here, the algorithm is rooted at $z_1^\delta$ instead of $\partial$, so we keep all reflected loops, and by construction, the loops that contain $z_1^\delta$ are not present anyway when performing Wilson's algorithm).

\subsection{Wired UST convergence}

Let $z_1, \ldots, z_n$ be $n$ points in $\Omega$ and for each sufficiently small $\delta$,  $z_1^{\d},\ldots, z_n^{\d}$ will denote the approximations of these points on $\d \Z^2$ (such that each $z_j^\delta$ is at distance at most $\delta$ from $z_j$). 

We will consider the uniform spanning tree with wired boundary conditions in $\Omega^\d$, the uniform spanning tree with free boundary conditions in $\Omega^\d$ and the uniform spanning tree in $\delta \Z^2$. 
We will denote by $\T_n^{f,\delta}$ and $\T_n^{w,\delta}$ the (smallest) finite subtrees of these spanning trees that contains $z_1^\delta, \ldots, z_n^\delta$.

Note that the tree $\T_n^{w,\delta}$ sometimes contains the boundary vertex $\partial$, so that it can also be viewed as a forest. We will
 denote the (possibly larger) tree that contains $z_1^\delta, \ldots, z_n^\delta$ and the boundary vertex $\partial$ by $\widehat \T_n^{w,\delta}$ (so the trees $\widehat \T_n^{w,\delta}$ and $\T_n^{wf,\delta}$ 
 are the trees that are constructed iteratively via Wilson's algorithm as described above). When $\partial \notin \T_n^{w,\delta}$, then $\widehat \T_n^{w,\delta}$ is 
the union of $\T_n^{w,\delta}$ with an additional branch that joins $\T_n^{w,\delta}$ to $\partial$. 

Our goal is to show that  these wired and free uniform spanning trees, with appropriately rescaled length-parametrization converges in distribution to their continuous SLE$_2$-tree counterparts in $\Omega$
with their natural parametrizations. Let us start with the wired boundary conditions: 

\begin{prop}[Wired UST convergence]\label{cor:lengthwiredUST}
The tree $\widehat \T_n^{w,\delta}$ parametrized by its Euclidean length (say from $\partial$) multiplied by a constant times $\d^{1/4}$ converges in law 
to its continuous counterpart $\widehat \T_n^{w}$ parametrized via its natural parametrization.
\end{prop}

Note that we need to multiply here by $\delta^{1/4}$ here instead of the usual $\delta^{5/4}$ because we consider the Euclidean distance on the tree instead of the graph distance. 

\begin{proof}
Let us define for each $j \le n$, the branch $\gamma_{j,\d}$ of the tree that joins $z_j^\d$ to the boundary. 
The joint law of $(\gamma_{1, \d}, \ldots, \gamma_{n, \d})$ converges in law to its continuous counterpart, for the weaker topology $\tau_w$  on simple paths up to time-reparametrization
(this follows the results in \cite {LSW_LERW} for the scaling limit of one single branch, noting that this result does not require an analytical boundary, so that one can apply it iteratively via Wilson's algorithm). 
Note also that the joint law of $(\gamma_{1, \d}, \ldots, \gamma_{n, \d})$  and of the meeting points of different $\gamma_{j, \d}$ converges as well to its continuous counterpart. 

Let us choose a sequence $\delta_k \to 0$, and couple all the trees $\widehat \T_n^{w,\delta_k}$ and $\widehat \T_n^w$ on a same probability space (via Skorokhod's representation theorem) so that for the topology $\tau_w$, 
$(\gamma_{1, \delta_k}, \ldots, \gamma_{n,\delta_k})$ converges almost surely to $(\gamma_1, \ldots , \gamma_n)$ (and that the meeting points between different branches converge as well). 

From Result \ref{res:convLERW}, for each $j$, $\gamma_{j, \delta_k}$ with appropriately rescaled length converges in law to $\gamma_j$ with its natural parametrization. 
It follows readily that in the previous coupling, $\gamma_{j, \delta_k}$ has to converge in probability to $\gamma_j$ for this stronger topology. 
This finally implies that the collection $(\gamma_{1, \delta_k}, \ldots, \gamma_{n,\delta_k})$ converges in probability to $(\gamma_1, \ldots , \gamma_n)$ (for this stronger topology), which implies the claim. 
\end{proof}

The previous proposition readily implies that: 
\begin {cor}
The tree $\T_n^{w,\delta}$ parametrized by its Euclidean length (say from $z_1^\d$) multiplied by a constant times $\d^{1/4}$ converges in law 
to its continuous counterpart $\T_n^w$ with its natural parametrization. 
\end {cor}
Indeed, using the same arguments as before, when $\T_n^{w,\delta} \not= \widehat \T_n^{w,\delta}$, the arm $\widehat \T_n^{w,\delta} \setminus \T_n^{w,\delta}$ is easily seen to converge to its continuous counterpart.

Before discussing the UST with free boundary conditions, let us first derive the convergence result for the whole-plane UST.
Here $\T_n^{\infty,\delta}$ denotes the smallest subtree of an UST in $\delta \Z^2$ that contains the points $z_1^\d, \ldots, z_n^\d$  on $\d \Z^2$ that are approximations of some points $z_1 \ldots, z_n \in \R^2$. 

\begin{prop}[Whole-plane UST]\label{prop:lengthplanarUST}
The tree $\T_n^{\infty,\delta}$ parametrized by its Euclidean length (say from $z_1^\d$) multiplied by a constant times $\d^{1/4}$ converges in law 
to its continuous counterpart $\T_{n}^\infty$ with its natural parametrization.
\end{prop}

\begin{proof}
This follows fairly directly from the previous result for the wired UST. Indeed, it is easy to see that 
for all $\eps>0$, one can find $R$ large enough so that for all $\delta$, the tree $\T_n^{\infty,\delta}$ can be coupled with the subtree $\T_n^{w,\delta}$ of the wired UST in the 
discrete approximation of the disk of radius $R$ around the origin, in such a way that they coincide with probability at least $1-\eps$ (this can be shown in the number of ways, for instance in the spirit of the proof 
of Lemma 3.1 in \cite {S0} using Wilson's algorithm and Beurling-type estimates such as Lemma 2.3 in \cite {L2}. One can first condition on the infinite
branch on the full-plane UST in $\delta \Z^2$ that starts near the point ${(R, 0)}$, and first note that with probability $1 - O(1)$  as $R \to \infty$ and uniformly with respect to $\delta$, 
this branch does not go through the ball of radius $R^{1/2}$ around the origin.
Then, using the same random walks in Wilson's algorithm for both, one can couple 
the two finite subtrees for the UST with wired conditions on this branch, with the UST with wired boundary condition in the discrete approximation of the ball of radius $R$, in such a way that 
the parts of the subtrees inside the balls of radius $R^{1/4}$ do coincide with probability $1 - O(1)$, again uniformly with respect to $\delta$. We safely leave those details to the reader). 
\end{proof}

\subsection{Free UST convergence} 
\label {A3}
We now turn to the more challenging case of the free UST.

\begin{prop}[Free UST]\label{prop:lengthfreeUST}
The tree $\T_n^{f,\delta}$ parametrized by its length (say from $z_1^\d$) renormalized by a constant times $\d^{5/4}$ converges in law 
to its continuous counterpart $\T_{n}^f$ with its natural parametrization.
\end{prop}

One important observation in order to deduce the results for the free UST from those for the wired UST is that in a free UST (and in its scaling limit)
the entire branch that joins any two given inner points (at positive distance from the boundary) will (in the scaling limit) remain at positive distance from $\partial \Omega$.
More precisely: 
\begin{lem}\cite[Theorem 11.1-(ii)]{S0}
\label {lemmadist}
For any given $z_1, \ldots, z_n$ in $\Omega$, for each $\eps$, one can find $r_0$ and $\delta_0$ so that for all $\d < \d_0$, 
\[
\P [ d( \T_n^{f,\delta}, \partial \Omega^\d) < r_0 ] \le \eps.
\]
\end {lem}

We now turn to the proof of Proposition \ref {prop:lengthfreeUST}. 
\begin{proof}[Proof of Proposition  \ref {prop:lengthfreeUST}]

We will prove this result by controlling the Radon-Nikodym derivative of the law of $\T_n^f$ with respect to that of the wired subtree $\T_n^w$. 

As noted above, for a given possible tree $T$:
\[
\P [\T_n^{f,\delta} = T]= \frac{1}{Z_n^{f,\delta}} 4^{- |T|} \times V_\d (T). 
\]
For the wired UST, we can compute the probability that $\T_n^{w,\delta} = T$ for the same tree $T$ by summing, over all 
possible additional simple branches $\gamma$ that connect $T$ to the boundary $\partial$, the probability that $\widehat \T_n^{w,\delta}$is equal to $T \cup \gamma$: 
\[
\P [\T_n^{w,\delta}  = T ] = 
\frac{1}{Z^{w,\delta}_n}\sum_{\gamma: T \leftrightarrow \partial \Omega^\delta} 
4^{-|\gamma|- |T|} U_\delta (T \cup \gamma).
\]
However, for each tree $T$, 
\[
\nu_{ \Omega^\delta \setminus T } \left( \{ e \ : \ e \hbox { joins } \partial \Omega^\d \hbox { and } T \} \right)
=
\sum_{\gamma: T \leftrightarrow \partial \Omega^\delta}  4^{-|\gamma|} \exp \left(\lambda_{\Omega^\delta \setminus T} \left( \{ \ell \ : \ \ell \cap \gamma \not= \emptyset \} \right)\right)
\]
(this corresponds to the decomposition of each excursion $e$ into its loop-erasure and the loops it encountered). 
Hence, as $U_\delta (T \cup \gamma) = U_\delta (T)  \exp \left(\lambda_{\Omega^\delta \setminus T} \left( \{ \ell \ : \ \ell \cap \gamma \not= \emptyset \} \right)\right)$, we have:
\[
\P [\T_n^{w,\delta} = T ]
=\frac{1}{Z^{w,\delta}_n} 4^{-|T|} U_\delta (T) \nu_{ \Omega^\delta \setminus T } \left( \{ e \ : \ e \hbox { joins } \partial \Omega^\d \hbox { and } T \} \right). 
\]

Our goal is to control the behavior of the ratio $ \P [\T_n^{f,\delta} = T ] / \P [\T_n^{w,\delta} = T ]$ as $\delta \to 0$, uniformly over all trees $T$ that stay at distance at least $r_0$ from the boundary of $\Omega$. 
The previous expressions show that this ratio is equal to  some constant $C=C(\delta, \Omega)$ (that is independent of $T$) times
\begin{equation}\label{eq:ratio}
\frac{V_\delta (T)}{U_\delta (T)}
=
\frac {   \exp\left(\lambda_{ \delta}^{r}\left(\{\ell \ :  \ z_1^\delta \notin \ell,  \ell \cap T \neq \emptyset, \ell \hbox { bounces off } \partial \Omega^\d \}\right)\right) } {
\nu_{ \Omega^\delta \setminus T } \left( \{ e \ : \ e \hbox { joins } \partial \Omega^\d \hbox { and } T \} \right) }
\end{equation}
because the mass of loops that do not bounce off the boundary appear in both expressions (for the free and wired trees) and cancel out. 

Let us first study the excursion measure term in the denominator of the right-hand side of (\ref{eq:ratio}). This is a well-known quantity that can be viewed as the discrete extremal distance between $\partial \Omega^\d$ and $T$. 
This quantity is uniformly close to its continuous Brownian counterpart (a.k.a. the extremal distance between $T$ and $\partial \Omega$) when $\d \to 0$. 
We will briefly explain in Section \ref {Sfinal} how to adapt the existing results 
for discrete extremal lengths of quadrilaterals to the present annular case. Let us also note that this quantity is bounded form below and from above 
uniformly with respect to all trees $T$ that are are distance greater than $r_0$ from $\partial \Omega$, as well as with respect to all $\delta$ small enough. Indeed, once $z_1$ and $z_2$ are fixed (say), then all excursions that go through some given tube from the boundary to itself and disconnects $z_1$ and $z_2$ in the domain, do necessarily intersect each of these trees, and the mass of this set of excursions converges as 
$\d \to 0$, which proves the lower bound. For the upper bound, one can use monotonicity of the extremal distance and bound this quantity by the 
extremal distance between the outer boundary of $\Omega$ and some cycle that is at distance smaller than $r_0$ from the boundary. 
 
We now turn to the study of the numerator of the right-hand side of  (\ref{eq:ratio}). Let us now look closer at the set of loops 
\begin{equation}\label{eq:set}
\{\ell \ :  \ z_1^\delta \notin \ell,  \ell \cap T \neq \emptyset, \ell \hbox { bounces off } \partial \Omega^\d \}
\end{equation}
appearing in (\ref{eq:ratio}).  
Let $w$ and $w'$ be two fixed disjoint concentric ($w'$ surrounds $w$) smooth curves that surround both $z_1$ and $z_2$ that both stay in the $r_0$-neighborhood of $\partial \Omega$, and denote 
their natural discretization on $\delta \Z^2$ by $w_\delta$ and $w_\delta'$. 
One can then decompose each loop $\ell$ in $\Omega^\delta$ that touches both 
$\partial \Omega^\delta$ and $T$ 
(where the tree $T$ stays at distance at least $r_0$ from $\partial \Omega$) into its upcrossings and downcrossings between $w_\delta$ and $w_\delta'$. More precisely, an upcrossing will correspond 
to a nearest-neighbor path from some point $y$ on $w_\delta$ to the first point $y'$ that belongs to $w_\delta'$, and a downcrossing will conversely 
correspond to a path (possibly reflected off the boundary of $\Omega^\delta$) that starts from some point $y'$ on $w_\delta'$ up to the first point $y$ that belongs to $w_\delta$. 
The maximal number $N(\ell)$ of disjoint subpaths of a loop $\ell$ that are upcrossings is finite, and is at least equal to $1$ (because the loop touches both $T$ and the boundary). Summing up, 
we get a decomposition of the loop into a concatenation of $N(\ell)$ pairs of paths, each pair consisting of an upcrossing from $w_\delta$ to $w_\delta'$ and one downcrossing from $w_\delta'$ to $w_\delta$. 
The upcrossings are usual random walk paths until their first hitting of $w_\delta'$ and the downcrossings are reflected random walk paths in $\Omega^\delta$ up to their first hitting of $w_\delta$. 

Let $r_1$ be the minimum of the distance $d(z_1, w)$ and $d( z_1, z_2)/ 2$. 
We say that an upcrossing from $w_\delta$ to $w_{\delta}'$ is good if it does not disconnect $z_1^\delta$ from the circle of radius $r_1$ around $z_1$, and if it does not go through the point $z_1^\delta$.
We say that a loop $\ell$ of (\ref{eq:set}) is good if all of its upcrossings are good. It is easy to see that the total mass of such good loops is finite, 
as the mass of good loops with exactly $n$ upcrossings decays exponentially in $n$. 

The loops that are not good intersect any tree that contains both $z_1^\delta$ and $z_2^\d$ (as $r_1$ is smaller than $d(z_1, z_2)/2$), so that the mass of the loops in (\ref{eq:set}) that are not good does not depend on $T$, and so can be incorporated into the constant $C(\delta,\Omega)$. 
Hence, it it enough to estimate the mass under $\lambda_\d^r$ of the set of loops
\begin{equation}\label{eq:setloops}
\{\ell \ :  \ \ell \cap T \neq \emptyset,\ \ell \hbox { is good, and bounces off } \partial \Omega^\d \}.
\end{equation}
This set can be decomposed according to the number $N(\ell)$ of upcrossings of a loop $\ell$, and according to the positions 
of the end-points $y_1, \ldots, y_{N(\ell)}$ and $y_1', \ldots, y_{N(\ell)}'$ of these up and down-crossings. More precisely, we can choose to root $\ell$ at the beginning of one of its upcrossings, chosen uniformly at random (so we 
weight this choice of a root by a factor $1/ N(\ell)$). The first upcrossing goes from $y_1$ to $y_1'$, the first downcrossing from $y_1'$ to $y_2$, and so on (the last downcrossing goes from $y_N'$ back to $y_1$).  
Recall that we also define $J(\ell)$ to be the maximal multiplicity of the loop (which will typically be equal to $1$). 
The total mass for $\lambda_\d^r$ of the set (\ref{eq:setloops}) can be written as: 
\begin{equation}\label{eq:sum}
o_{\delta}(1) + 
\sum_{N \ge 1} \sum_{y_1, \ldots,y_N} \sum_{y_1', \ldots , y_N' } N^{-1} R_\d ( y_1', y_2; y_2', y_3; \ldots, y_N', y_1) \times Q_\d^g ( y_1, y_1'; y_2, y_2'; \ldots ; y_N, y_N') 
\end{equation}
where 
\[
R_\d ( y_1', y_2; y_2', y_3; \ldots, y_N', y_1) := p_\d^r (y_1', y_2) \ldots p_\d^r (y_N', y_1 ),
\]
and $p_\d^r (y', y)$ denotes the probability that a reflected random walk in $\Omega^\d$ started from $y'$ hits $w_\d$ at $y$, and 
where $Q_\d^g$ denote the probability that $N$ independent random walks started from $y_i$, $i =1, \ldots, N$ respectively, hit $w_\d'$ at $y_i'$ respectively, 
that they are all good upcrossings, and that at least one of them intersects the tree $T$. 

The $o_{\delta}(1)$ term in (\ref{eq:sum}) is due to the fact that we overcounted here the set of loops that have a multiplicity $J(\ell)$ that is not equal to $1$. 
It is easy to see that this term is negligible as $\delta\to 0$ as the contribution of loops of non-trivial multiplicity vanishes (for instance, note that $J \ge 2$ implies that for some $j \not= k$, $y_j = y_{k}$). 

To conclude, we need to control the limit of (\ref{eq:sum}), i.e., we need to control, for each given $N$, the limit of the sum of the products $R_\d \times Q_\d^g$.
Note that $R_\d$ involves reflected random walks but deals only with discrete harmonic measures, 
while $Q_\d^g$ involves usual (non-reflected) random walks, but requires control on the trajectories, because of the condition that the upcrossings are good and that at least one of them hits $T$.

Let us first look at the quantity $R_\d$. For each point $y_i' (\delta) \in \d \Z^2$, we can view $p_\d^r (y_i', \cdot )$ as the distribution of the first point at which the reflected random walk in $\Omega^\d$ started from $y_i'(\d)$ hits $w_\d$. 
On the one hand,  the dependence 
on the starting point $y_i'(\d)$ can be controlled uniformly via a simple coupling argument
(two random walks started from nearby points can be mirror-coupled in such a way that they meet with 
high probability before exiting some ball). 
On the other hand, when $\d \to 0$, if $y_i' (\d) \to y_i$, then this discrete  harmonic measure 
can be shown to converge to its continuous counterpart, as we briefly explain in Subsection \ref {Sfinal}.

Let us now focus on $Q_\d^g$. 
Just as for $R_\d$
(but with the roles of $y_i'$ and $y_i$ exchanged), for each set of points $(y_1, \ldots, y_N)$ in $\d \Z^2$, one can view $Q_\d^g$ as a random measure $\mu_{\d, T}$ on the set of points $(y_1', \ldots, y_N')$. 

Let us define for each $y(\d) \in \d \Z^2$ the measure $\mu_\d (y (\d), \cdot)$ on endpoints of good upcrossings that start from $y (\d)$ and the measure  $\tilde \mu_\d (y (\d), \cdot)$ on endpoints of good upcrossings that start from $y(\d)$ and do not intersect $T$. 

When $\d \to 0$, when the starting point $y(\d)$ converge to some $y$, and when the tree $T=T(\delta)$ converges to some continuous tree
(noting that all points on this continuous tree are regular for Brownian motion simply because it is connected by arcs, so that as soon as 
Brownian motion hits the tree for the first time at some point $z$, it also  disconnects this point $z$ from infinity immediately after that time), 
the usual convergence in distribution of simple random walk to Brownian motion implies that the measures $\mu_d (y(\d), \cdot)$ 
and $\tilde \mu_\d (y (\d), \cdot)$ converge to their Brownian counterpart (the previous observation ensures that the scenario where the limiting Brownian motion hits the tree much before the random walk does 
is unlikely). 
Since the quantity $\mu_{\d,T}$ is a linear combination of finite products of these measures, it follows readily that it  also converges to its Brownian counterpart when $(y_1 (\delta), \ldots, y_n(\delta))$ 
converges to some $(y_1, \ldots , y_n)$ (and the tree $T(\delta)$ converges to some continuous tree). 

Furthermore (for instance by simple coupling arguments), one can 
again see that these measures depend continuously on 
the starting points (and also uniformly with respect to $\d$ -- the total variation distance between the two measures for two sets of discrete starting points goes to $0$ when the distance 
between these starting points goes to $0$).  

From these convergence of the measures $R_\d$ and $\mu_{\d, T}$ (uniform in the starting point parameters, and in $T$ when $d(T, \partial \Omega)$ is bounded from below), the convergence of the sum (\ref{eq:sum}) finally follows 
(indeed, the contribution to (\ref{eq:sum}) of loops with more than $N$ upcrossings is exponentially small, as noted previously).

We can note also that the same arguments (just evaluating the mass of good loops without reference to $T$) shows that the 
masses of the set of loops in (\ref {eq:setloops}) are all bounded uniformly (with respect to all $T$'s that are at distance greater than $r_0$ of the boundary,  and all $\d$ small enough).
Hence, on the set of trees that stay at distance at least $r_0$ from the boundary, 
the Radon-Nikodym derivative between the laws of the free and wired subtrees converges uniformly. Combining this with Lemma \ref {lemmadist} and  the convergence in law of the finite-subtrees of the wired UST, we can conclude  follows from that of the finite subtrees of the free UST do also converge in law. 
\end{proof}

\subsection{Joint convergence of the UST and its dual}

Finally, we consider the joint convergence of a wired UST with its dual free UST. 
We let $\T^\delta$ denote  the whole wired UST in the discretization $\Omega^\delta$ of a bounded simply-connected domain $\Omega$ with analytic boundary. 
Its dual $\T^{\dagger\delta}$ is then a free UST in some subgraph of $\delta (\Z + (1/2))^2$, which also approximates $\Omega$.

Again, for each $z_1, \ldots, z_n$, we can define the finite subtree tree $\T_n^\delta$ of $\T^\delta$. We also consider 
the smallest finite subtree of the dual tree $\T^{\dagger,  \delta}$ that connects
some approximations of $z_1, \ldots, z_n$ in the dual graph $\delta (\Z + (1/2))^2$. We denote this tree by $\T_n^{\dagger,  \delta}$. 
\begin{cor}\label{prop:joint}
The pair $(\T_n^\delta , \T_n^{\dagger,  \delta})$ (with appropriately rescaled length) converges in distribution to its continuous counterpart (with the natural parametrizations). 
\end{cor}

\begin{proof}
The tree and dual tree are known to jointly converge (as unparametrized trees), and \cite[Remark 10.14]{S0} shows that their limits are deterministic functions of each other. So we can readily deduce that 
$(\T_n^\delta , \T_n^{\dagger,  \delta})$ converges to its continuous counterpart, in the sense of unparametrized trees (i.e., trees up to monotone reparametrization). 
In order to deduce the convergence of the parametrized trees, we can use the same argument as before: 
\begin {itemize} 
 \item We consider any given sequence $\d_k \to 0$ and we couple the pairs $(\T_n^{\delta_k} , \T_n^{\dagger,  \delta_k})$ for all $k$, so that this pair converges almost surely as $k \to \infty$ (as unparametrized trees). 
 \item Then, we note that when one looks at the sequence $\T_n^{\delta_k}$ alone, Proposition \ref{cor:lengthwiredUST} then implies that it converges in probability for the stronger topology of parametrized trees. 
 Similarly, Proposition \ref {prop:lengthfreeUST} shows that $\T_n^{\dagger,  \delta_k}$ converges in probability for the stronger topology. It follows readily that the pair converges in probability for the stronger topology 
to the pair of limiting trees parametrized by their natural parametrization (noting that the natural parametrization is a deterministic function of the tree). 
\end {itemize}
\end{proof}

Similarly, using Proposition \ref{prop:lengthplanarUST}, we obtain the following corollary for the subtrees of the whole-plane UST and its dual (with obvious notation): 
\begin{cor}\label{prop:joint2}
The pair $(\T_n^{\infty, \delta} , \T_n^{\infty \dagger,  \delta})$ (with appropriately rescaled length) converges in distribution to its continuous counterpart (with the natural parametrizations). 
\end{cor}

\subsection{Final estimates} 
\label {Sfinal} 
We now briefly explain how to adapt the existing proofs in the literature in order to obtain the two results that we have used in Section \ref {A3} -- the following arguments are in the spirit of 
\cite {Ch}: 

\begin {itemize} 
 \item Let us first discuss the mass $m_A$ of the set of discrete excursions that cross a discrete conformal annulus $A$ (we applied this to the term in the denominator of (\ref {eq:ratio})). 
 For each given $\delta$, let us consider the harmonic function $\widehat H^\delta$ in $A$ that takes value $0$ on the inner boundary and $1$ on the outer boundary. 
For any closed simple loop $L$ on the dual lattice of $\delta \Z^2$ that separates the two boundary components of the conformal annulus, we can define 
the flow $\phi (L)$ of the gradient vector field $\nabla \widehat H^\delta$ through $L$ (defined as the sum of the increments of $\widehat H^\delta$ over all 
properly oriented edges dual to those of $L$). We can note that 
(a) The flow $\phi (L)$ does in fact not depend on the choice of $L$ because $\widehat H^\delta$ is harmonic. 
(b) When one chooses $L$ to be ``outer boundary circuit'', one gets exactly the quantity $4m_A $, by definition of the excursion measure and the representation of $\widehat H^\delta$ via random walk hitting probabilities, the $4$ factor comes from the $1/4$ contribution in $m_A$ of the first outgoing jumps away from the outer boundary, across the edge of the boundary circuit). 
Hence, the quantity $m_A$ is in particular equal to $\phi (L_0^\delta)$ through some given $L_0^\delta$, which is the approximation on the lattice dual to $\d \Z^2$ of a given loop $L_0$ 
that is at positive distance of the outer boundary.  

To conclude, we can then note that the harmonic function $\widehat H^\delta$ and its (appropriately defined) discrete derivatives converge to their continuous counterpart $\widehat H$ as $\d \to 0$, and uniformly when considered 
  at some positive distance from the boundary of $A$ (and therefore uniformly on a neighborhood of $L_0$). Hence, we can conclude the quantity $4m_A$ converges to the flow of the gradient
  of $\widehat H$ through $L_0$, 
  which is a positive and finite quantity (note that no normalization is needed as $\phi (L_0)$ will typically be the sum of $O (\delta^{-1})$ terms of order $\delta$), 
  and that in turn can be interpreted in terms of the mass of a set of continuous Brownian excursions, using similar arguments.

  \item Let us now discuss the convergence of the harmonic measure $p_\d^r (y', \cdot)$ for a random walk reflected on the outer boundary of a conformal annulus $A$. 
 It suffices to show the convergence of the harmonic measure (seen from $y'$) of a given subarc $a$ of the inner boundary of $A$. 
 This is a discrete harmonic bounded function $H^\d$ of the starting point $y'$. It therefore has subsequential limits as $\d \to 0$. A limit $h$ of a convergent subsequence is necessarily harmonic, and by Beurling-type estimates like Lemma 2.3 in \cite {L2}, it has 
 boundary values $0$ and $1$ on the inner boundary of the annulus ($1$ of the arc $a$ and $0$ on its complement). To conclude, we need to show that $h$ has Neumann boundary conditions on the outer boundary of $A$. 
 
To do this, let us consider the discrete harmonic conjugate $H^{\d*}$ of $H^\d$, normalized so that it takes the value $0$ at some given interior point. These harmonic conjugates are bounded uniformly with respect to $\d$, so that they also have subsequential limits. Hence, by extracting a further subsequential limit, one gets a joint convergence of $(H^\d, H^{\d*})$ to some pair $(h, h^*)$.  As above, Beurling-type estimates imply that $h^*$ is constant on the outer boundary, which is to say that $h$ has Neumann boundary conditions on the outer boundary. 
 
\end {itemize}

\bigbreak

\begin {center}
 {\sc Acknowledgements.}
\end {center}

\noindent
We acknowledge support and/or hospitality of the following grants and institutions: the Einstein Foundation Berlin, TU Berlin, SNF-155922 and 175505, NCCR Swissmap and the FIM at ETH Z\"urich, the 
ESPRC under grant EP/103372X/1, and the Isaac Newton Institute. The authors thank Dima Chelkak for his advice concerning Section \ref {Sfinal}. 
L.D. also thanks G\'abor Pete for useful discussions. We are also extremely grateful for the many insightful comments of an anonymous referee.

\end{document}